\documentclass[12pt]{amsart}

\usepackage{dsfont}
\usepackage{amssymb}
\usepackage{fullpage}
\usepackage{xypic}
\usepackage{graphicx}
\usepackage{color}
\usepackage{mathtools}
\usepackage[usenames,dvipsnames]{xcolor}
\usepackage[latin1]{inputenc}
\usepackage{indentfirst}
\usepackage{epstopdf}
\usepackage{subfigure}
\usepackage{emptypage}
\usepackage{enumitem}
\usepackage{hyperref}
\hypersetup{
    colorlinks=true, 
    linktoc=all,     
    linkcolor=blue,  
    citecolor=JungleGreen,
}

\newtheorem{theorem}{Theorem}

\newtheorem{corollary}[theorem]{Corollary}
\newtheorem{proposition}[theorem]{Proposition}
\newtheorem{lemma}[theorem]{Lemma}
\theoremstyle{definition}

\newtheorem{definition}[theorem]{Definition}
\newtheorem{remark}[theorem]{Remark}

\newcommand{\interior}{\mathrm{int}}

\newcommand{\id}{\mathrm{id}}

\newcommand{\p}{\mathfrak{p}}
\newcommand{\q}{\mathfrak{q}}
\newcommand{\e}{\mathfrak{e}}
\renewcommand{\c}{\mathfrak{c}}
\newcommand{\R}{\mathds{R}}

\renewcommand{\O}{\mathcal{O}}
\newcommand{\Z}{\mathds{Z}}
\newcommand{\N}{\mathds{N}}
\newcommand{\D}{\mathds{D}}
\newcommand{\Q}{\mathds{Q}}
\newcommand{\T}{\widehat{T}}
\newcommand{\A}{\mathds{A}}
\renewcommand{\AA}{\widetilde{\mathds{A}}}
\renewcommand{\L}{\mathds{L}}
\renewcommand{\P}{\mathds{P}}
\newcommand{\X}{\widetilde{X}}
\newcommand{\F}{\widehat{F}}
\newcommand{\acc}{\mathcal{A}}

\newcommand\restr[2]{{
  \left.\kern-\nulldelimiterspace 
  #1 
  \right|_{#2} 
  }}

\title[Accessible points rotate as prime ends]{Accessible points rotate as prime ends in backward or forward time}

\author{Luis Hern\'{a}ndez--Corbato}


\address{
   Luis Hernández-Corbato\\
   Instituto de Ciencias Matemáticas CSIC-UAM-UC3M-UCM\\
   C/ Nicolás Cabrera 13--15\\
   Spain}
\email{luishcorbato@mat.ucm.es}

\thanks{
The author is supported by MINECO (grants SEV-2015-0554 and MTM2015-63612-P) and benefitted from a Juan de la Cierva-Formaci\'on grant from the 2014 call (code FJCI-2014-21020) whose effective date was postponed until June 2016. The author expresses its gratitude to ICMAT for its early reception in 2016 during which most of this worked was carried out.
}

\begin{document}

\maketitle

\begin{abstract}
Let $f$ be a homeomorphism of $\overline{\A}$, the closed annulus, isotopic to the identity and
let $X$ be a closed $f$--invariant subset of $\overline{\A}$ whose complement is homeomorphic to the half--open annulus.
The dynamics in the circle of prime ends of the complement of $X$ has an associated rotation number, $\rho$.
We prove that either the rotation number of all forward semi--orbits of accessible points of $X$ are well--defined
and equal to $\rho$ or the rotation number of all backward semi--orbits of accessible points of $X$
is well--defined  and equal to $\rho$.
The result generalizes to the open and half--open annulus provided the semi--orbits under consideration are relatively compact.
\end{abstract}

\section{Introduction}
Every open connected and simply connected subset $V$ of the Riemann sphere which misses at least two points
is conformally equivalent to $\D = \{z \in \mathds{C}: |z| < 1\}$ by the Riemann mapping theorem.
In general, the conformal map $\phi\colon \D \to V$ cannot be extended to $\partial \D$, not even continuously, because the topology
of the boundary of $V$ can be very wild. In fact, the extension is possible if and only if $\partial V$
is locally connected and the extension is injective iff $\partial V$ is a Jordan curve.
Whilst proving this result, Caratheodory \cite{caratheodory} developed the theory of \emph{prime ends}. This theory associates an ideal
boundary $\P(V)$, the circle of prime ends of $V$, to every simply connected planar proper domain $V$ and provides
$\widehat{V} = V \sqcup \P(V)$, the prime end compactification of $V$,
with a topology that makes it homeomorphic to $\overline{\D}$.
Points in $\partial \D$ are identified to prime ends of $V$
but prime ends do not correspond in general to points in $\partial V$.

Any homeomorphism $\varphi\colon \D \to V$ can be extended to a homeomorphism between $\overline{\D}$ and $\widehat{V}$.
Evidently, if $\varphi$ extends to a point $z \in \partial \D$ there is an arc
$\gamma_z \colon [0,1] \to V \cup \{\varphi(z)\}$ which lands at $\gamma_z(1) = \varphi(z)$.
The converse result is true provided $\varphi$ is conformal.
The set of points $x \in \partial V$ for which there exists an arc $\gamma_x$ contained in $V$ except from $x$, its
landing point, are called \emph{accessible}.
At the same time, such an arc $\gamma_x$ determines a unique prime end in $\P(V)$.
Prime ends and accessible points are thus very closely related.
The article establishes a link between these objects in terms of rotation in annular dynamics.

The notion of \emph{rotation number} of circle homeomorphisms introduced by Poincaré was generalized
first to circle degree--1 endomorphisms and then to annular and toral dynamics (see \cite{misiurewicz}).
A great amount of research has been conducted around this concept.
Given a homeomorphism of the (open, closed or half--open) annulus $f \colon A \to A$ isotopic to the identity $F$ a lift of $f$ to
$\widetilde{A}$, the universal cover of $A$, and $x \in \widetilde{A}$, the limit of $(F^n(x))_1/n$ as $n \to +\infty$
does not always exist ($(\cdot)_1$ denotes the value of the lift of the angular coordinate in the cover).
Therefore, the notion of rotation number of a point is not well--defined for every point in $A$.
The rotation interval of $x$ is the set of limit points of sequences of the form
$\{((F^{n_i}(x))_1 - (x)_1)/n_i\}_i$, where $n_i$ tends to $+\infty$. Similarly, the rotation interval of
an $f$--invariant continuum $X \subset A$ is defined as the set of limit points of sequences
$\{((F^{n_i}(x_i))_1 - (x_i)_1)/n_i\}_i$, where $x_i$ belongs to the lift of $X$.
The rotation interval is, indeed, an interval which, as opposed to the circle homeomorphisms case, may be non--degenerate.
The convexity follows from an equivalent description of the rotation interval adapted from \cite{misiurewicz}:
the rotation interval is the set of values $\int_{A} u \, d\mu$, where $u$ is the displacement
function ($u(z) = (F(\tilde{z}))_1 - (\tilde{z})_1$ for any lift $\tilde{z}$ of $z$)
and $\mu$ ranges over all $f$--invariant Borel probability measures supported on $X$.

An essential annular continuum $X \subset \overline{\A}$ is a continuum whose complement $\overline{\A} \setminus X$
has exactly two connected components, $U_+$ and $U_-$, which are homeomorphic to the half--open annulus $\A$.
If $X$ does not properly contain any other essential annular continuum it is called a circloid.
Barge and Gillette \cite{barge} proved the following realization result for invariant circloids with empty interior (cofrontiers):
any rational number in the rotation interval of $X$ is realized as the rotation number of a periodic orbit in $X$.
Koropecki \cite{koropecki} has recently generalized this theorem to any circloid.
Furthermore, the periodic orbit can be chosen from $\partial X$ provided the rotation interval is non--degenerate
\cite{koropasseggi}.

Using prime end theory we can compactify each annular domain $U_{\pm}$ with a circle $\P(U_{\pm})$ and the result is a closed annulus.
A homeomorphism which leaves $U_{\pm}$ invariant induces homeomorphisms in $\P(U_{\pm})$,
whose associated rotation numbers, $\rho_{\pm}$, are called prime end rotation numbers.
Matsumoto \cite{matsumoto} (see an alternative proof in \cite{nagoya})
proved that these numbers belong to the rotation interval of $X$.
However, even cofrontiers may have non--degenerate rotation intervals so one cannot expect to learn much
of the dynamics within $X$ just from the prime end rotation numbers.

A strong relationship with respect of their rotation numbers could be expected between accessible points and prime ends.
The kind of domains $U$ in which we locate our discussion are annular domains such $U_{\pm}$ in the preceeding paragraphs.
More precisely, $U$ will be an open and proper subset of $\A$, invariant under an orientation--preserving homeomorphism and homeomorphic to $\A$. We will denote its complement by $X$, $U = \A \setminus X$.
The question that motivates this article is whether the rotation number of an accessible point $p \in \partial U = \partial X$, if defined, is somehow related to the rotation number of a prime end determined by an arc in $U$ landing at $p$ under the induced action $\widehat{f}$ on the circle of prime ends $\P(U)$.
On one hand, since $\widehat{f}$ is a circle homeomorphism, the latter rotation number
is independent of the choice of prime end and is called \emph{prime end rotation number} and will be denoted $\widehat{\rho}(F, X) \in \R$. It is computed in the
universal cover of the set of prime ends with the induced dynamics that is coherent with $F$, a prescribed lift of $f$.
On the other side of the potential link, the rotation number $\lim_{n\to+\infty} (F^n(p))_1/n$ may not exist even if $p$ is an accessible point, see Section \ref{sec:horseshoe}.
Surprisingly, the mismatch is solved by broadening our scope and considering both the backward and the forward orbit of $p$
in the search for potential rotation numbers related to $\widehat{\rho}(F, X)$.
Whenever they exist, the limits of $(F^n(p))_1/n$ as $n \to -\infty$ and $n \to +\infty$ will be called
\emph{backward} and \emph{forward rotation number} of $p$, respectively. They are the rotation numbers associated
to $\O^-(p)$, the backward orbit of $p$, and $\O^+(p)$, the forward orbit of $p$, respectively.

A concrete example in which to study the preceeding question are Birkhoff attractors (see Le Calvez \cite{lecalvez}).
These attractors are essential annular continua that appear when we consider a certain class of dissipative twist maps in the annulus.
For an open subclass of these maps, the corresponding Birkhoff attractor $\Lambda$ has non--trivial rotation interval
whose endpoints are the prime end rotation numbers $\rho_- < \rho_+$ associated to their complementary regions.
Generically, every rational number $a/b \in (\rho_-, \rho_+)$ is realized by a hyperbolic periodic point $p_{a/b} \in \Lambda$,
none of which is accessible.
Consider one of these periodic points $p_{a/b}$.
Its stable manifold $W^s(p_{a/b})$ meets $\Lambda$ in infinitely many accessible points.
The same is true in the universal cover for $\widetilde{\Lambda}$ the lift of $\Lambda$ and any choice of lift $\widetilde{p_{a/b}}$ of $p_{a/b}$.
Every $x \in W^s(\widetilde{p_{a/b}}) \cap \widetilde{\Lambda}$ satisfies
$0 = \lim_{n \to +\infty} (F^n(x))_1 - (F^n(p_{a/b}))_1$ so $\lim_{n \to +\infty} (F^n(x))_1/n - a/b = 0$. Thus, we find
plenty of accessible points in $\widetilde{\Lambda}$ whose forward rotation number is different from the prime end rotation number.
In this paper we prove (see Theorem \ref{thm:dicotomia}) that these points have well--defined backward
rotation number equal to the prime end rotation number.

Trivially, backward and forward rotation numbers of periodic points coincide.
It is natural to wonder whether accessible points with different rotation numbers can coexist in the boundary of invariant domains. Cartwright and Littlewood \cite{cartwright} answered in the negative for open topological disks in $S^2$ except for the possible existence of a fixed point result of the coalescence of periodic points. This exception disappears in planar dynamics, see \cite{yorke}.
Recently, Passeggi, Potrie and Sambarino \cite{uruguayos} have proven the uniqueness of the rotation number
of periodic points under a weaker form of accessibility in their way to showing that if the rotation interval
associated to an invariant attracting cofrontier is non--degenerate the dynamics has positive topological entropy.

The main result of this paper shows that either backward or forward rotation numbers
of accessible points are always well--defined and equal to the prime end rotation number.
The theorem requires relative compactness
on the orbits under consideration without which rotation may not even be well--defined.
We use the term \emph{bounded} to refer to orbits in the universal cover whose projection to the annulus is relatively compact.
Notice that in the case of the closed annulus $\overline{\A}$ the compactness assumption on orbits is automatically satisfied.

\begin{theorem}\label{thm:dicotomia}
Let $f \colon \A \to \A$ be an orientation--preserving homeomorphism
leaving invariant a non--empty closed set $X$ such that $\A \setminus X$ is homeomorphic to $\A$.
Fix $F$ a lift of $f$ to $\AA$, denote $\X$ the lift of $X$
and $\rho = \widehat{\rho}(F, X)$ the prime end rotation number of $F$ in $X$.
Then, one of the following statements holds:
\begin{itemize}
\item For every accessible point $p \in \X$ with bounded backward semi--orbit
$$\lim_{n \to -\infty} \frac{(F^n(p))_1}{n} = \rho.$$
\item For every accessible point $q \in \X$ with bounded forward semi--orbit
$$\lim_{n \to +\infty} \frac{(F^n(q))_1}{n} = \rho.$$
\end{itemize}
\end{theorem}

In brief, the theorem tells us that accessible points behave roughly as expected as far as rotation is concerned, at least
in one time direction.
Linear drifts from the expected rotational behavior are only allowed either in backward time or in forward time,
not simultaneously.
In particular, if a point has well--defined rotation number, that is, backward and forward rotation numbers exist and coincide
then they are equal to the prime end rotation number.
Our discussion on the asymptotic behavior of accessible points goes a bit further.
The results obtained in this work are summarized in the next theorem.

\begin{theorem}\label{thm:resumen}
Under the notation and the hypothesis of the previous theorem:
\begin{enumerate}
\item \label{item:1} There are no accessible points $p, q \in \X$ such that $\O^-(p)$ and $\O^+(q)$ are bounded and
$$
\lim_{n \to -\infty} (F^n(p))_1 - n\rho = - \infty\enskip \text{and} \enskip \lim_{n \to +\infty} (F^n(q))_1 - n\rho = + \infty.
$$
$$
\text{or} \enskip \lim_{n \to -\infty} (F^n(p))_1 - n\rho = + \infty\enskip \text{and} \enskip \lim_{n \to +\infty} (F^n(q))_1 - n\rho = - \infty.
$$
\item \label{item:2}  Suppose that the sequences $\{(F^n(p))_1 - n \rho\}_{n \le 0}$
and $\{(F^n(q))_1 - n \rho\}_{n \ge 0}$ are unbounded for some accessible points $p, q \in \X$ such that $\O^-(p)$ and $\O^+(q)$ are bounded. Then:
\begin{enumerate}
\item \label{item:a} Both sequences $\{(F^n(p))_1\}_{n \le 0}$ and $\{(F^n(q))_1\}_{n \ge 0}$ have the form $n\rho + o(|n|)$.
Equivalently, the backward rotation number of $p$ is equal to the forward rotation number of $q$ and they are both equal to $\rho$:
$$\lim_{n \to -\infty} (F^n(p))_1/n = \lim_{n \to +\infty} (F^n(q))_1/n = \rho.$$
\item \label{item:b} If $\rho \in \Q$ then the closure of both the projection onto $\A$ of $\O^-(p)$ and of $\O^+(q)$
contain a periodic point whose rotation number is $\rho$.
\item \label{item:c} If $\rho \in \Q$ and $\{(F^n(q))_1 - n \rho\}_{n \ge 0}$ converges either to $-\infty$ or to $+\infty$
then any limit point of the projection onto $\A$ of the forward orbit of $q$ is periodic and has rotation number $\rho$.
An analogous statement holds for $p$ and the projection of $\O^-(p)$.
\end{enumerate}
\end{enumerate}
\end{theorem}

Item (\ref{item:1}) in Theorem \ref{thm:resumen} deals with the case we named ``transverse'' drift:
the forward and backward semi--orbit drift from the expected one in different directions.
This would be precisely the case of a point $x \in \X$ whose full orbit has a well--defined rotation number $\rho_x = \lim_{n \to \pm\infty} (F^n(x))_1/n \in \R$
different from $\widehat{\rho}(F, X)$. Theorem \ref{thm:resumen} (\ref{item:1}) concludes that $x$ is not accessible.
The topological notion of relative winding number has been developed ad--hoc to give a neat proof of this statement.

Theorem \ref{thm:resumen} (\ref{item:2}) deals with the general case. It is proved in Section \ref{sec:mainresults}.
There, extra hypothesis (H1) and (H2) are introduced to rule out
the coexistence of forward and backward semi--orbits drifting from the expected one. These hypothesis are
automatically satisfied when one of the drifts is, at least, linear or the semi--orbits do not accumulate entirely
in the set of periodic points with rotation $\rho$ provided $\rho = \widehat{\rho}(F, X) \in \Q$.

The following result is an easy corollary of Theorem \ref{thm:resumen} (\ref{item:a}) and, in turn, yields Theorem \ref{thm:dicotomia} as
a corollary as well.

\begin{corollary}
Suppose $\{(F^n(p))_1 - n \rho\}_{n \le 0}$ is unbounded for some accessible point $p \in \X$ with bounded backward semi--orbit.
Then, every accessible point $q \in \X$ with bounded forward semi--orbit has well--defined forward rotation number equal to $\rho$, i.e.
$$\lim_{n \to +\infty} (F^n(q))_1/n = \rho.$$

Similarly, if the sequence $\{(F^n(q))_1 - n \rho\}_{n \ge 0}$ is unbounded for some accessible point $q \in \X$ with $\O^+(q)$ bounded
then $\lim_{n \to -\infty} (F^n(p))_1/n$ exists and is equal to $\rho$ for every accessible point $p \in \X$
such that $\O^-(p)$ is bounded.
\end{corollary}

Another corollary of the results proved in this article has been suggested by A. Koropecki:

\begin{corollary}\label{cor:rhorealizado}
Let $f \colon \overline{\A} \to \overline{\A}$ be a homeomorphism isotopic to the identity of the closed annulus. As usual, denote $X$ an invariant continuum such that $\overline{\A} \setminus X$ is homeomorphic to $\A$, $F$ and $\X$ lifts of $f$ and $X$ to the universal cover and $\rho$ the prime end rotation number of $F$ in $X$.

Then, there exists a point $q$ in the boundary of $\X$ whose forward rotation number is equal to $\rho$. Furthermore, there exist an ergodic $f$--invariant Borel probability measure supported in $\partial X$ whose rotation number is $\rho$.
\end{corollary}

Recall that (for a fixed choice of lift $F$) the rotation number of an $f$--invariant measure $\mu$ is $\rho(F, \mu) = \int_A u \,d\mu$, where $u$ denotes the displacement function defined above.

The choice of domain in which the dynamics is placed, typically the half--open annulus $\A$, is not very relevant, as we discuss in the next section. Prime ends and related concepts are introduced in Section \ref{sec:primeends}. The notion of relative winding number between arcs and between prime ends is explained in Section \ref{sec:windingnumber} and used in the subsequent section to discuss the case of accessible periodic points.
The rotational horseshoe is described in Section \ref{sec:horseshoe} as an example in which
accessible points may not have well--defined forward rotation number.

Sections \ref{sec:transverse} and \ref{sec:mainresults} are the core of the paper. They contain the proof of Theorem \ref{thm:resumen}; more precisely, the proofs of (\ref{item:1}) and (\ref{item:2}), respectively. In Section \ref{sec:transverse}, the relative winding number is exploited successfully in a short argument that addresses (\ref{item:1}). The latter section needs more preliminary technical work until we reach the heart of the proof: Propositions \ref{prop:notinfront} and \ref{prop:haciaatrasrotanigual}. It ends with a proof of Corollary \ref{cor:rhorealizado}. The sharpness of the results is supported by the constructions described by Theorems \ref{thm:exampletransversal} and \ref{thm:counterexamplebouncing}.

\section{Setting}\label{sec:setting}

The half--open annulus $S^1 \times (-\infty, 0]$ is denoted $\A$. The one point compactification of $\A$, obtained by adding the
point $e^-$ associated to the lower end of the annulus, is homeomorphic to the closed unit disk $\overline{\D}$.
The map $\pi\colon \AA = \R \times (-\infty, 0] \to \A$ which sends
$(\widetilde{\theta}, r)$ to $(e^{2\pi i \widetilde{\theta}}, r)$ is a universal covering of $\A$.
The dynamics is generated by an orientation--preserving homeomorphism $f\colon \A \to \A$.
This map is automatically isotopic to the identity and restricts to an orientation--preserving circle homeomorphism in $\partial \A$.
Any lift $F \colon \AA \to \AA$ of $f$ is also orientation--preserving and isotopic to the identity.
The article focuses on the dynamics of an $f$--invariant non--empty closed set $X \subset \A$ such that
$U = \A \setminus X$ is homeomorphic to $\A$.
The reason for these assumptions is that we can then compactify $U$ with a circle (of prime ends) as is described in Section \ref{sec:primeends}.
An equivalent characterization of $X$ would be: $X$ is a closed proper subset of $\A$
that is adherent to $e^-$, does not meet $\partial \A = S^1 \times \{0\}$,
does not have compact connected components and such that $\A \setminus X$ is connected.

The choice of the half--open annulus as workplace is not very relevant.
Any dynamics in the closed annulus easily fits in our setting. If $g$ is a homeomorphism of $S^1 \times [-1, 0]$ isotopic to the identity
we can insert the dynamics it generates in $\A$ and trivially extend it to the whole annulus. If $K$ is a $g$--invariant continuum
that does not intersect $S^1 \times \{0\}$ and
such that its complement is homeomorphic to $\A$ then $X = K \cup S^1 \times (-\infty, -1]$ satisfies
the same properties in the half--open annulus.

Suppose now that $h$ is a homeomorphism of the open annulus $S^1 \times \R$ isotopic to the identity and $X$ is a closed
$h$--invariant set which is adherent to the lower end $e^-$ but not to the upper end.
Then, for large $r$ we can take a global isotopy $\{I_t\}_{t=0}^1$, $I_0 = \id$,
which leaves a neighborhood of $X$ fixed and satisfies $I_1(h(S^1 \times \{r\})) = S^1 \times \{r\}$.
In order to examine dynamical properties of $X$
we can cut off $S^1 \times [r, +\infty)$ and work with the map $I_1 \circ h$, which is equal to $h$ in a neighborhood of $X$.

Let $K$ be a non--separating planar continuum invariant under an orientation--preserving planar homeomorphism $f$.
Suppose $K$ contains at least two points.
Again, composing $f$ with a suitable isotopy of $\R^2$ which is equal to the identity around $K$, we can assume $f$ leaves invariant
a large disk that contains $K$. The complement of $K$ in that disk is homeomorphic to $\A$.
By Cartwright--Littlewood Theorem \cite{cartwright}, $K$ contains a fixed point $p$.
We can puncture the disk at $p$ and work with the induced dynamics in the resulting half--open annulus
that leaves invariant the closed set $X = K \setminus \{p\}$.

An important remark is due in order to genuinely accept the previous constructions in our work.
The prime end rotation number associated to these sets remains invariant under the modifications proposed.
This is a direct consequence of the definitions
in Section \ref{sec:primeends} because the set of prime ends is not modified in any step.
A general argument addressing this technical details on prime ends can be found in Section 3 of \cite{duke}.

Even though we already saw that some other typical settings adjust well to the half--open annulus,
there is an important issue concerning $\A$ which must be addressed:
the rotation number of orbits that diverge towards $e^-$ may not be well--defined.
The problem is that a change of angular coordinate in $\A$ leads to a different value of the limit that defines the rotation number.
For example, consider the
fibered rotation $h_{\varphi} \colon (\theta, r) \mapsto (\theta + \varphi(r), r)$ in $\A$. Let $p \in \A$ be a point
whose forward orbit under a map $f$ converges to the lower end. If $\varphi$ is not bounded,
the rotation number of $p$ under $f$ and the rotation number of $h_{\varphi}(p)$ under
$h_{\varphi} \circ f \circ h_{\varphi}^{-1}$ may be different. A clear exposition of this phenomenon
and how to navigate through it can be found in Le Roux \cite{leroux}.

The previous paragraph explains why we will only consider orbits which stay away from the lower end of $\A$.
Obviously, this is true for every orbit when we insert a dynamics of the closed annulus into $\A$.
However, in general we require a compactness hypothesis on the orbits under consideration.

\begin{definition}
Let $F \colon \AA \to \AA$ be a lift of an orientation--preserving homeomorphism of $\A$.
An orbit or a semi--orbit under the action of $F$ is said to be \emph{bounded} if its projection onto $\A$ is relatively compact or,
equivalently,  if it is contained in the lift of a closed annulus in $\A$.
\end{definition}

Our work concerns accessible points and prime ends associated to $X$.
These objects are defined from the complement of $X$ so, in a certain way, the spirit of this paper is to
study the dynamics of an invariant set from the outside.
For this purpose, we will constantly make use of arcs in $\A \setminus X$.
In this work, the term \emph{arc} (resp.\emph{ half--open arc}) is used to refer to injective maps $\gamma\colon I \to \A$,
where $I = [0,1]$ (resp. $I = [0,1)$),
but mainly also to refer to their images $\gamma(I)$, which will be usually denoted $\gamma$ as well.

\section{Theory of prime ends}\label{sec:primeends}

Carath\'eodory proved that a conformal map $\phi\colon \D \to V$ extends to a homeomorphism between $\overline{\D}$ and $\overline{V}$
if and only if $\partial V$ is locally connected.
As part of his research on conformal mappings, he presented the theory of prime ends in \cite{caratheodory}.
To any given open, connected and simply connected subset $V \subset S^2$ such that $S^2 \setminus V$ contains at least two points
one can associate the set of prime ends $\P(V)$ and the union $V \sqcup \P(V)$ is given a topology that makes it homeomorphic to
a closed disk.
The set $\P(V)$ equipped with the subset topology is then homeomorphic to a circle and is frequently referred to as
the circle of prime ends of $V$.
A complete account on the theory from the viewpoint of conformal mappings is found in \cite{col}\cite{pommerenke}.

Mather presented in \cite{mather} a purely topological approach to prime ends
that works for more general open subsets on surfaces.
Following \cite{mather},
the theory of prime ends will be now briefly presented in the particular case
where the open set $U$ is a proper subset of the half--open annulus $\A$,
contains $\partial \A$ and is homeomorphic to $\A$. Our setting essentially fits the original one because after adding
the lower end of $\A$ and a disk capping $\partial \A$ we are left with a connected open and simply connected domain $V$ of $S^2$ whose complement
contains more than one point. It is easy to see that the set of prime ends associated to $U$ as will be defined below
is equal to the classical circle of prime ends associated to $V$.

 A \emph{cross--cut} $c$ of $U$
is an arc whose endpoints lie in $X = \A \setminus U$ and is otherwise contained in $U$.
The cross--cut $c$ splits $U$ into two connected components,
one of them, $V(c)$, is homeomorphic to an open disk and the other one is homeomorphic to $\A$.
A sequence $\{c_n\}_{n = 1}^{\infty}$ of pairwise disjoint cross--cuts such that $V(c_{n+1}) \subset V(c_n)$
is called a \emph{chain} (of cross--cuts). A chain $\{c'_n\}_n$ is said to divide another chain $\{c_n\}_n$ if for every $i$
there exists $j$ such that $V(c'_j) \subset V(c_i)$. Two chains are \emph{equivalent} if any of them divides the other.
A \emph{prime end} $\p$ is an equivalence class in the set of chains that contains a minimal element, i.e. a
chain $\{c_n\}_n$ such that any other chain that divides $\{c_n\}_n$ is equivalent to $\{c_n\}_n$ (hence it belongs to the $\p$).
The set of prime ends of $U$ is denoted $\P(U)$.

Given a cross--cut $c$ of $U$, denote $\widehat{V}(c)$ the union of $V(c)$ and the set of prime ends represented by
chains $\{c_n\}_n$ such that $V(c_n) \subset V(c)$ for every $n$. The disjoint union $U \sqcup \P(U)$ is given a topology
whose basis is composed of the open subsets of $U$ and the sets of the form $\widehat{V}(c)$ for some cross--cut $c$ of $U$.
The main theorem in the theory adapted to this setting states that the topological space $\widehat{U} = U \sqcup \P(U)$,
called the \emph{prime end compactification} of $U$, is homeomorphic to the closed annulus.

\subsection{Accessible prime ends}

The \emph{principal set} $\Pi(\p)$ of a prime end $\p$ consists of the points which are the limit of a sequence of cross--cuts
whose diameter tends to 0 in a chain that represents $\p$.

As stated in the introduction,
a point $p \in \partial X = \partial U$ is said to be \emph{accessible} (from $U$) if there exists an arc $\gamma\colon [0,1] \to U \cup \{p\}$
such that $\gamma[0,1) \subset U$ and $\gamma(1) = p$. Accessible points are dense in $\partial U$.
From the definition of prime end it can be shown that $\gamma$ determines a unique prime end $\p$
in the sense that given any representative $\{c_n\}$ of $\p$, $\gamma^{-1}(V(c_n))$ contains an interval of the form $(t, 1)$.
Furthermore, $\p$ satisfies $\Pi(\p) = \{p\}$.
Conversely, if $\Pi(\p) = \{p\}$ then $p$ is accessible from $U$: it is possible to construct an arc $\gamma\colon [0,1] \to U \cup \{p\}$ such
that $\gamma^{-1}(p) = \{1\}$ and the prime end determined by $\gamma$ is $\p$.
In this case we say $p$ is the \emph{principal point} of $\p$.
The previous equivalence explains the reason why prime ends for which their principal set is a singleton
will be called \emph{accessible} throughout this paper. In the literature they are
often called prime ends of first or second kind (depending on whether their impression $\cap_n\overline{V(c_n)}$ is a singleton or not
for any chain $\{c_n\}$ that represents it).
Accessible prime ends form a dense set in $\P(U)$.
Indeed, a basis of the topology of $\P(U)$ is given by the sets of the form $\widehat{V}(c) \cap \P(U)$, where $c$ is a cross--cut of $U$.
To construct an accessible prime end in these sets it suffices to take any closed arc completely contained in $\widehat{V}(c)$
that determines a prime end in the sense explained above.
Notice that an accessible point in the boundary of $U$ may determine several different accessible prime ends.

Given a half--open arc $\gamma \colon [0,1) \to U$, we can both view it as an arc in $U \subset \A$
and as an arc in the prime end compactification of $U$, $\widehat{U}$. This perspective yields an
alternative statement to the result used in the previous paragraph to compare accessible points and accessible prime ends.
If $\gamma(t)$ converges as $t \to 1$ to a point $p$ in $X$ then it can also be extended as an arc in the prime end compactification:
there is a $\p \in \P(U)$ such that $\gamma(t)$ as viewed in $\widehat{U}$ converges to $\p$ as $t \to 1$.
In this case, $\p$ is an accessible prime end whose principal point is $p$.
However, in general the converse is not true: not every half--open arc in $U$ that converges to an accessible prime end can be extended to
a closed arc in $\A$.

\subsection{Line of prime ends}\label{subsec:lineend}

Recall that $\AA = \R \times (-\infty, 0]$, $\pi \colon \AA \to \A$ is the universal cover
and write $\X = \pi^{-1}(X)$, $\widetilde{U} = \pi^{-1}(U)$.
The goal of this subsection is to lift the prime end compactification obtained for $U \subset \A$ to the universal cover $\widetilde{U} \subset \AA$.
The result will be an infinite band bounded by $\partial \AA$ and a line composed of what will be called prime ends of $\widetilde{U}$,
denoted $\L(U)$.
It is important to remark that $\L(U)$ is not the result of applying prime end theory directly to $\widetilde{U} \subset \AA$,
should it be somehow possible, but, instead, $\L(U)$ contains the lifts of genuine prime ends of $U$.

Some of the previous terminology is now redefined in this new setting.
Define a cross--cut $\widetilde{c}$ of $\widetilde{U}$ as the lift to $\AA$ of a cross--cut $c$ of $U$.
Then, $\widetilde{c}$ separates $\widetilde{U} = \AA \setminus \X$ in two components, only one of which does not contain $\partial \A$,
say $\widetilde{V}(\widetilde{c})$, and $\pi^{-1}(V(c)) = \cup_{\pi(\widetilde{c}) = c} \widetilde{V}(\widetilde{c})$.
Then, we define $\L(U)$ as the set of minimal equivalence classes of chains of cross--cuts.
The notation chosen will be explained later.
The elements of $\L(U)$ will be called prime ends of $\widetilde{U}$
although they are not obtained by applying prime end theory to $\widetilde{U}$.

The union $\widetilde{U} \sqcup \L(U)$ is given a topology through a basis composed of the open sets of $\widetilde{U}$ and the sets of the form $\widehat{\phantom{\,}\widetilde{V}\phantom{\,}}(\widetilde{c})$ defined as the union of $\widetilde{V}(\widetilde{c})$ and the
subset of prime ends represented by chains $\{\widetilde{c}_n\}_n$ such that $\widetilde{V}(\widetilde{c}_n) \subset \widetilde{V}(\widetilde{c})$ for every $n$.
Since $\pi$ restricted to $\widetilde{V}(\widetilde{c})$ is injective, it induces
a covering map $\widehat{\pi}\colon \L(U) \to \P(U)$.
Moreover, the generator of the deck transformations $T\colon \AA \to \AA$ given by $T(\widetilde{\theta}, r) = (\widetilde{\theta} + 1, r)$ also
induces a generator of the deck transformations $\widehat{T}\colon \L(U) \to \L(U)$.
Therefore, $\widehat{\pi}\colon \L(U) \to \P(U)$ is a universal cover
and we call $\L(U)$ the \emph{line of prime ends} of $\widetilde{U}$, $\L(U)$ is homeomorphic to $\R$.
The set $\widehat{\phantom{\,}\widetilde{U}\phantom{\,}} = \widetilde{U} \cup \L(U)$ will be called
\emph{prime end closure} of $\widetilde{U}$.
Let us remark the fact that the prime end closure is not compact.
Topologically, it is an infinite closed band whose boundary components are $\partial \AA$
and $\L(U)$.

The notions of accessible prime end and arc that determines a prime end translate verbatim to $\L(U)$ and $\widetilde{U}$.
Henceforth, we will work with a special class of arcs:

\begin{definition}
An arc $\gamma\colon [0,1] \to \AA$
 is called a \emph{hanging arc} provided $\gamma(t) \in \partial \AA =  \R \times \{0\}$ if and only if $t = 0$.
The point $\gamma(1)$ is said to be the \emph{landing point} of $\gamma$.

A hanging arc $\gamma \colon [0,1] \to \AA$ such that $\gamma[0,1) \subset \widetilde{U}$ and $\gamma(1) \in \X$ is called \emph{lead line}.
\end{definition}

The term hair was used in \cite{nagoya} instead of lead line, but we think \emph{lead line} expresses better the idea
of an arc used to thoroughly examine all hidden cavities and corners of $\widetilde{U}$.
An example of lead line is $\gamma_{\p}$ from Figure \ref{fig:primeends}.
Note that a lead line determines a prime end in $\L(U)$ and separates $\widetilde{U}$ in exactly two connected components.

\begin{figure}[htb]
\centering
\begin{tabular}{l c r}
\includegraphics[scale = 0.7]{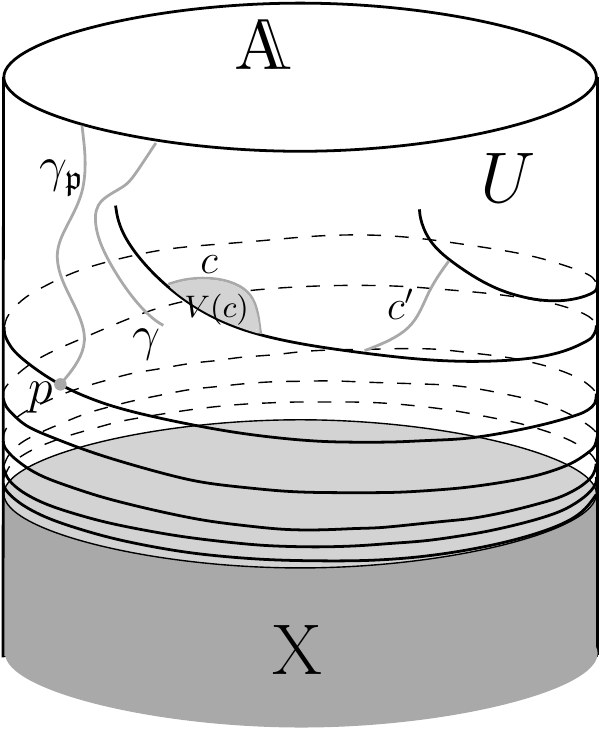} & \hspace{1cm} & \includegraphics[scale =
0.7]{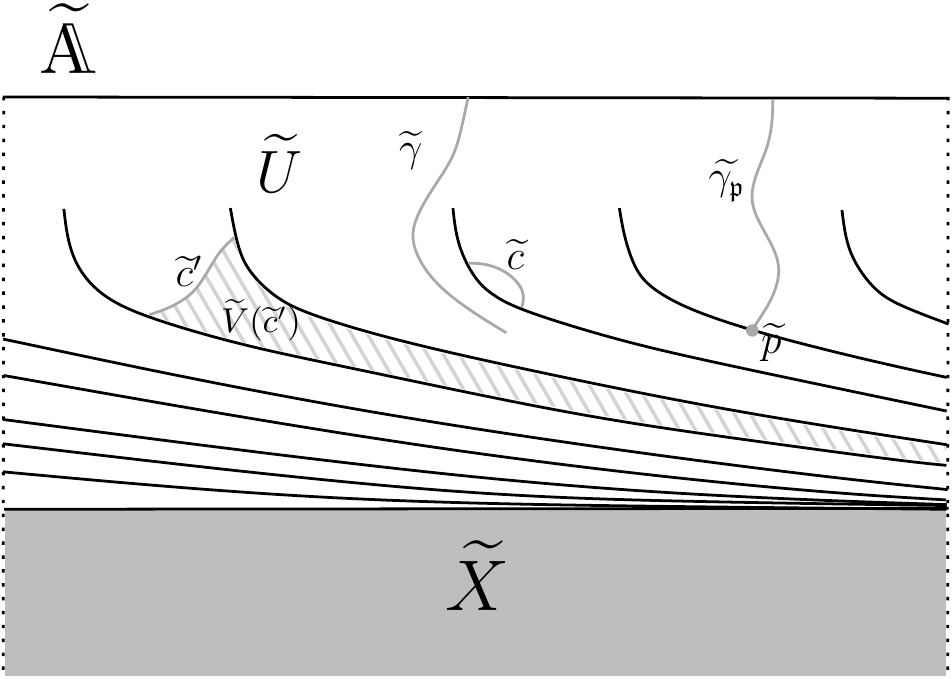}
\end{tabular}
\caption{$X$ is composed of two strings wrapped around $\A$ that accumulate in a limit circle and the cylinder left below (in dark gray).
$\gamma_{\p}$ is a lead line that lands at $p$, $\gamma$ is a hanging arc.
Cross--cuts $\widetilde{c}, \widetilde{c}'$ are lifts of $c, c'$. Note that $\widetilde{V}(\widetilde{c})$
is compact, whereas $\widetilde{V}(\widetilde{c}')$ is not.}
\label{fig:primeends}
\end{figure}

The next lemma contain some results that are analogues of classical facts of prime end theory, see Mather \cite{mather}.

\begin{lemma}\label{lem:primeends1}
\begin{enumerate}
\item \label{item:leadlinesdisjoint} Given any two different prime ends $\p, \q \in \L(U)$,
there are lead lines $\gamma_{\p}, \gamma_{\q}$ that determine $\p$ and $\q$,
 respectively, and are disjoint (except for their landing points in the case $\Pi(\p) = \Pi(\q)$).
\item \label{item:homotopicleadlines} Any pair of lead lines $\gamma, \gamma'$ that determine
the same prime end $\p$ are homotopic through lead lines that also
determine $\p$.
\end{enumerate}
\end{lemma}

\subsection{Order in $\L(U)$}

Since $\L(U)$ is homeomorphic to $\R$, it can be equipped with a total order relation, denoted $\prec$, in a way that the topology of $\L(U)$
(as a subspace of the prime end closure) coincides with the order topology generated by $\prec$.
Among the two possible choices of order of $\L(U)$, we adopt the one that is coherent with $\widehat{T}$. Thus, we have $\p \prec \widehat{T}(\p)$ for every $\p \in \L(U)$.


Lead lines are useful to translate the order relation in $\L(U)$ to topological terms.
The order relation $\prec$ is illustrated as follows: given two different accessible prime ends $\p, \q \in \L(U)$,
$\p \prec \q$ if and only if there are disjoint (except for the possible coincidence of landing points) lead lines $\gamma_{\p}, \gamma_{\q}$ that determine $\p, \q$, respectively,
and such that $(\gamma_{\p}(0))_1 < (\gamma_{\q}(0))_1$.
If the previous conditions hold it follows that the inequality is also true for any pair of disjoint (with the possible exception of a common landing point) lead lines.

The notion of convergence to $+\infty$ in $\L(U)$ can be defined as follows:

\noindent
A sequence $\{\e_n\}_{n = 0}^{\infty} \subset \L(U)$ tends to $+\infty$ if for every $\p \in \L(U)$ there exists $n_0$ such that $\p \prec \e_n$
holds for every $n \ge n_0$.

The convergence to $-\infty$ is defined in an analogous fashion.

\begin{lemma}\label{lem:primeends2}
Let $\e_2 \prec \e_1 \prec \e_0$ be accessible prime ends in $\L(U)$ and $\gamma_2, \gamma_1, \gamma_0$ lead lines that determine
$\e_2, \e_1, \e_0$, respectively. Suppose $\gamma_1 \cap \gamma_0 = \emptyset$ and $\gamma_2 \cap \gamma_0 \neq \emptyset$.
Then, $\gamma_2 \cap \gamma_1 \neq \emptyset$.

As a consequence, if $\{\e_n\}$ is a decreasing sequence of accessible prime ends that diverges to $-\infty$, there are pairwise disjoint
lead lines $\{\gamma_n\}$ such that $\gamma_n$ determines $\e_n$.
\end{lemma}
\begin{proof}
The first statement follows from the fact that $\e_0$ and $\e_2$ are separated by $\gamma_1$ in the prime end closure of $\widetilde{U}$.

The second assertion can be proved by induction. Take $\gamma_0$ to be any lead line that determines $\e_0$.
Assume, as inductive hypothesis, that there are pairwise disjoint lead lines $\gamma_0, \ldots, \gamma_n$ that determine $\e_0,\ldots,\e_n$, respectively.
Use Lemma \ref{lem:primeends1} to obtain a lead line $\gamma_{n+1}$ disjoint to $\gamma_n$ that determines $\e_{n+1}$.
The first part of this lemma concludes that $\gamma_{n+1}$ is also disjoint to the rest of lead lines and the inductive step is proved.
\end{proof}

\subsection{Prime end rotation number}

Let $f\colon \A \to \A$ be an orientation--preserving homeomorphism and fix a lift $F$ of $f$ to $\widetilde{\A}$.
Assume $X$ is invariant under $f$, hence so is $U$. The map $f$ extends to a homeomorphism $\widehat{f}\colon \widehat{U} \to \widehat{U}$
as follows: if a minimal chain $\{c_n\}$ represents $\p$ then $\widehat{f}(\p)$ is the prime end defined by the minimal chain $\{f(c_n)\}$.
There is a unique lift $\widehat{F}$ of $\widehat{f}$ to $\widehat{\phantom{\,}\widetilde{U}\phantom{\,}}$,
the prime end closure of $\widetilde{U}$, which coincides with $F$ in $\widetilde{U}$.

Given a circle homeomorphism $g \colon S^1 \to S^1$ and a fixed lift $G \colon \R \to \R$ of $g$, the rotation
number of $G$ is defined as $\lim_{n \to +\infty} G^n(x)/n$ for any $x \in \R$.
Notice that the limit is equal to $\lim_{n \to -\infty} G^n(x)/n$.

The \emph{prime end rotation number}, $\widehat{\rho}(F, X)$, is defined as
the rotation number of the restriction of $\F$ to $\L(U)$.
Note that $\widehat{\rho}(F, X)$ only depends on the choice of lift $F$ of $f$ up to an integer constant.
Since the speed of rotation of every orbit of a circle homeomorphism is equal, we can describe $\widehat{\rho}(F, X)$
by looking at the orbit of one prime end of $\L(U)$, say $\p$.
As a reference, take a prime end $\q$ not in the orbit of $\p$.
For any $n \in \Z$, there exists $m = m(n)$ such that $\widehat{T}^m(\q) \prec \widehat{F}^n(\p) \prec \widehat{T}^{m+1}(\q)$.
Clearly, $\widehat{\rho}(F, X) = \lim_{n \to +\infty} \frac{m(n)}{n}$.

This approach is particularly interesting if we restrict our attention to accessible prime ends. Suppose $\q$ is accessible.
Let $\beta$ be a lead line which determines the prime end $\q$ and projects onto a simple arc in $\A$, i.e. $\beta \cap T \beta = \emptyset$.
Denote by $D$ the domain in $\widetilde{U}$ enclosed by $\beta$ and $T\beta$.
Notice that $(D)_1$ may not be bounded but $D$ is relatively compact if viewed as a subset of the prime end closure.
An elementary planar topology argument shows that the orbit of $\beta$ is composed of pairwise disjoint arcs.
The domains $T^k(D)$, $k \in \Z$, are the connected components of the complement of the orbit of $\beta$ in $\widetilde{U}$,
$T^k(D)$ being adherent to $T^n(\beta)$ and $T^{n+1}(\beta)$.
A remarkable fact from which we will take advantage throughout the paper is the following:

The inequality $\widehat{T}^m(\q) \prec \widehat{F}^n(\p) \prec \widehat{T}^{m+1}(\q)$ is equivalent to the existence of
a lead line $\gamma$ which determines $\widehat{F}^n(\p)$ and is contained in $T^m(D)$ except for its landing point.

\section{Relative winding number}\label{sec:windingnumber}

Consider the following four points in $\AA = \R \times (-\infty, 0]$: $(0,0), (1,0), (0,-1)$ and $(1,-1)$.
We can draw hanging arcs $\gamma$ joining $(0,0)$ to $(1,-1)$ and $\gamma'$ joining $(1,0)$ to $(0,-1)$ that are disjoint.
The most simple examples can be easily splitted into two categories depending on whether $\gamma$ goes ``under'' $\gamma'$ or viceversa.
This intuition extends to more complicated drawings where at first sight it is not as easy to deduce which arc goes ``under'' the other.
See Figure \ref{fig:under}.
This sort of classification is well--posed in a topological setting where the answer does not vary after
a continuous deformation of the picture, relative to the endpoints.

\begin{figure}[htb]
\begin{center}
\includegraphics[scale = .6]{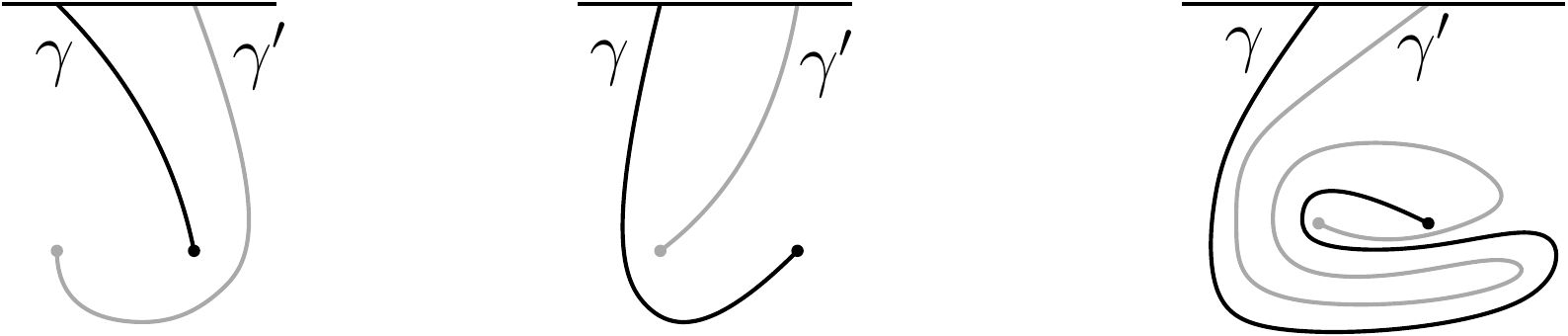}
\end{center}
\caption{On the left, $\gamma'$ (gray) goes ``under'' $\gamma$ (black). In the middle, $\gamma$ goes ``under'' $\gamma'$. On the right
it is not as clear, but if we are allow to deform continuously the picture to straighten $\gamma$ (with endpoints fixed) we see
$\gamma'$ arc going ``under'' $\gamma$.}
\label{fig:under}
\end{figure}

The idea behind Theorem \ref{thm:resumen} (\ref{item:1}), which was the origin of this article, is to detect the presence
of the picture described in the previous paragraph in some complicated rotational dynamics and play with it
until we reach a contradiction.
The notion that we introduce to formalize the idea of arcs going ``under'' another is the relative winding number
between pairs of hanging arcs.

Let $\gamma, \gamma'$ be two hanging arcs such that $(\gamma(0))_1 < (\gamma'(0))_1$ and $\gamma(t) \neq \gamma'(t)$ for every $t \in [0,1]$.
Consider $\phi\colon [0,1] \to S^1$ defined by
$$\phi(t) = \frac{\gamma'(t) - \gamma(t)}{||\gamma'(t) - \gamma(t)||}.$$
Notice that $\phi(0) = 1$. Let $\R \to S^1: x \mapsto e^{i \pi x}$ be the universal covering projection.
The map $\phi$ has a unique lift $\widetilde{\phi}\colon [0,1] \to \R$ such that $\widetilde{\phi}(0) = 0$.

\begin{definition}
The \emph{relative winding number} of the hanging arcs $\gamma$ and $\gamma'$ is defined by
$$w(\gamma, \gamma') := \widetilde{\phi}(1) \in \R.$$
\end{definition}

\begin{remark}\label{rmk:reparameterize}
We may replace $\phi$ in the definition with any map $\phi'$ homotopic to $\phi$ relative to $\{0, 1\}$. Indeed,
the unique lift $\widetilde{\phi}'$ of $\phi'$ such that $\widetilde{\phi}'(0) = 0$ satisfies $\widetilde{\phi}'(1) = \widetilde{\phi}(1)$.
In particular, if $\varphi_s\colon[0,1] \to [0,1]$ is a continuous family of reparameterizations of $\gamma$, $\varphi_0 = \id$,
it follows that $w(\gamma, \gamma') = w(\gamma \circ \varphi_1, \gamma')$ as long as
the relative winding number is well--defined for every pair $(\gamma \circ \varphi_s, \gamma')$,
that is $\gamma(\varphi_s(t)) \neq \gamma'(t)$ for every $s, t \in [0,1]$.
\end{remark}

Given $a \in \R$, $a \notin 1/2 + \Z$, denote $[a]$ the closest integer to $a$.
The following lemma shows that if we move the hanging arcs $\gamma, \gamma'$ their relative winding number does not oscillate significantly
as long as their landing points do not exchange sides.

\begin{lemma}\label{lem:endhomotopy}
Let $\{\gamma_s\}_{s=0}^1, \{\gamma'_s\}_{s=0}^1$ be two homotopies of hanging arcs such that
$$(\gamma_s(0))_1 < (\gamma'_s(0))_1, \enskip
\gamma_s(t) \neq \gamma'_s(t) \enskip \text{and} \enskip
(\gamma_s(1))_1 \neq (\gamma'_s(1))_1$$
for all $s,t \in [0,1]$. Then,
$$ [w(\gamma_0, \gamma'_0)] = [w(\gamma_1, \gamma'_1)].$$
Furthermore, if the homotopies of hanging arcs fix the landing points, then
$$w(\gamma_0, \gamma'_0) = w(\gamma_1, \gamma'_1).$$
\end{lemma}
\begin{proof}
The condition in the statement implies
$$\phi_s(1) = \frac{\gamma'_s(1) - \gamma_s(1)}{||\gamma'_s(1) - \gamma_s(1)||} \neq e^{i\pi/2}, e^{-i\pi/2},$$
hence $\widetilde{\phi_s}(1) = w(\gamma_s, \gamma'_s) \notin 1/2 + \Z$, for every $s \in [0,1]$. The result trivially follows.
\end{proof}

In brief, the previous lemma states that
 if the landing points remain fixed by the homotopies the relative winding number is constant as long as it is defined.

\begin{lemma}\label{lem:updownhairs}
Let $\gamma, \gamma'$ be disjoint hanging arcs such that $(\gamma(0))_1 < (\gamma'(0))_1$.
\begin{itemize}
\item If $(\gamma'(1))_1 < \min (\gamma)_1$ then $[w(\gamma, \gamma')] = -1$.
\item If $\max (\gamma')_1 < (\gamma(1))_1$ then $[w(\gamma, \gamma')] = 1$.
\end{itemize}
\end{lemma}
\begin{proof}
For the first statement, define $\phi_r\colon [0, 1] \to S^1$ as follows

\[
\phi_r(t) =
\begin{dcases}
\frac{\gamma'(2t) - \gamma(0)}{||\gamma'(2t) - \gamma(0)||} & \text{if} \enskip 0 \le t \le 1/2 \\
\frac{\gamma'(1) - \gamma(2t-1)}{||\gamma'(1) - \gamma(2t-1)||} & \text{if} \enskip 1/2 \le t \le 1.\\
\end{dcases}
\]

\begin{figure}[htb]
\begin{center}
\includegraphics[scale = 0.8]{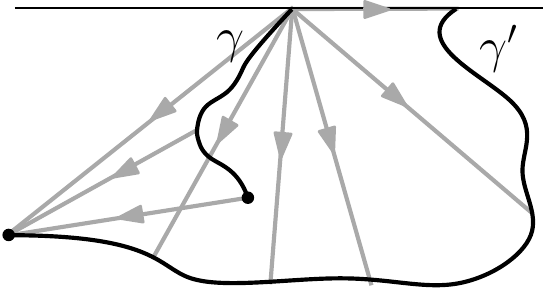}
\end{center}
\caption{Gray arrows display the map $\phi_r$.}
\label{fig:relativewindingnumber}
\end{figure}

See Figure \ref{fig:relativewindingnumber}.
It is easy to check, see Remark \ref{rmk:reparameterize}, that $\phi_r$ is homotopic to $\phi$ relative to $\{0,1\}$.
Thus, its lift $\widetilde{\phi_r}$ satisfies $\widetilde{\phi_r}(1) = \widetilde{\phi}(1) = w(\gamma, \gamma')$.
From the definition it is easy to see that $\widetilde{\phi_r}(t) \in (-1, 0)$ for $0 < t < 1/2$.
Since $(\gamma'(1))_1 < (\gamma(0))_1$, $\widetilde{\phi_r}(1/2) \in (-3/2, -1/2)$.
Moreover, $(\gamma'(1))_1 < (\gamma(t))_1$ for every $t$, so $\widetilde{\phi_r}(t) \in (-1/2, -3/2)$ for all $1/2 \le t \le 1$.
Then,
\[
[w(\gamma, \gamma')] = [\widetilde{\phi_r}(1)] = -1.
\]

The second statement is proven in a similar fashion.

%
\end{proof}

\subsection{Accessible prime ends}

Given two different accessible prime ends $\p, \q \in \L(U)$ whose principal points are $p \neq q$,
Lemma \ref{lem:primeends1} (\ref{item:leadlinesdisjoint}) guarantees the existence of disjoint lead lines $\gamma_{\p}, \gamma_{\q}$
such that $\gamma_{\p}(1) = p, \gamma_{\q}(1) = q$.
Recall also Lemma \ref{lem:primeends1} (\ref{item:homotopicleadlines}):
 any two lead lines that determine $\p$ are homotopic through lead lines that also determine $\p$.

\begin{definition}
Let $\p, \q \in \L(U)$, $\p \prec \q$, be accessible prime ends with different principal points
and let $\gamma_{\p}, \gamma_{\q}$ be lead lines as above.
The \emph{relative winding number} of $\p, \q$ is defined by
$$w(\p, \q) = w(\gamma_{\p}, \gamma_{\q}).$$
\end{definition}
\begin{proposition}
$w(\p, \q)$ does not depend on the choice of $\gamma_{\p}, \gamma_{\q}$.
\end{proposition}
\begin{proof}
Let $\gamma'_{\p}, \gamma'_{\q}$ be another pair of lead lines satisfying the properties above.
We have to show that $w(\gamma_{\p}, \gamma_{\q}) = w(\gamma'_{\p}, \gamma'_{\q})$.
Note that $(\gamma_{\p}(0))_1 < (\gamma_{\q}(0))_1$ and $(\gamma'_{\p}(0))_1 < (\gamma'_{\q}(0))_1$.

Assume first that $(\gamma_{\p}(0))_1 < (\gamma'_{\q}(0))_1$.
Let $\{\gamma^s_{\p}\}_{s=0}^1$ be a homotopy of lead lines between $\gamma^0_{\p} = \gamma'_{\p}$ and $\gamma^1_{\p} = \gamma_{\p}$
that fixes the landing point $p$ and avoids $\gamma'_{\q}(0)$.
Similarly, let $\{\gamma^s_{\q}\}_{s=0}^1$ be a homotopy of lead lines between $\gamma^0_{\q} = \gamma'_{\q}$
and $\gamma^1_{\q} = \gamma_{\q}$ such that $\gamma^s_{\q}(1) = q$ and $\gamma^s_{\q}(0) \neq \gamma_{\p}(0)$ for every $s$.
Since the arcs $\gamma^s_{\p}$ are disjoint to $q$, if $\epsilon > 0$ is small enough we have that
\[ \gamma^s_{\p}[0,1] \cap \gamma'_{\q}[1-\epsilon, \epsilon] = \emptyset
\enskip \enskip \text{ and also that } \enskip \enskip
\gamma^s_{\p}[0, \epsilon] \cap \gamma'_{\q}[0,1] = \emptyset.
\]

\noindent
Likewise, we may shrink $\epsilon > 0$ so that additionally satisfies
\[
\gamma^s_{\q}[0, \epsilon] \cap \gamma_{\p}[0,1] = \emptyset
\enskip \enskip \text{and} \enskip \enskip
\gamma^s_{\q}[0,1] \cap \gamma_{\p}[1-\epsilon, \epsilon] = \emptyset.
\]

The idea is to use $\epsilon$ to construct a reparameterization $\varphi$ of $[0,1]$ so that $w(\gamma^s_{\p}, \gamma'_{\q}\circ \varphi)$
and $w(\gamma_{\p}, \gamma^s_{\q} \circ \varphi)$ is well--defined for every $s$.
It suffices that $\varphi[0,\epsilon] \supset [0, 1-\epsilon]$. Then,
\[
w(\gamma'_{\p}, \gamma'_{\q}) = w(\gamma'_{\p}, \gamma'_{\q} \circ \varphi) = w(\gamma_{\p}, \gamma'_{\q} \circ \varphi) =
w(\gamma_{\p}, \gamma_{\q} \circ \varphi) = w(\gamma_{\p}, \gamma_{\q}),
\]
where the first and last equality comes from Remark \ref{rmk:reparameterize} and the other two from Lemma \ref{lem:endhomotopy}.

In order to conclude the proof, notice that if the hypothesis $(\gamma_{\p}(0))_1 < (\gamma'_{\q}(0))_1$ does not hold
then $(\gamma'_{\p}(0))_1 < (\gamma_{\q}(0))_1$ and the procedure can be reversed: first use the homotopy between $\gamma'_{\q}$
and $\gamma_{\q}$ and then a suitable reparameterization of the homotopy between $\gamma'_{\p}$ and $\gamma_{\p}$.
\end{proof}

Our immediate goal is to translate Lemma \ref{lem:updownhairs} to the language of accessible prime ends.

\begin{definition}
For an accessible prime end $\p \in \L(U)$, set
\[
\min \p = \sup_{\gamma \in \Gamma(\p)} \min (\gamma)_1, \enskip \enskip \enskip
\max \p = \inf_{\gamma \in \Gamma(\p)} \max (\gamma)_1,
\]
where $\Gamma(\p)$ denotes the set of lead lines $\gamma$ which determine the prime end $\p$.
\end{definition}

Clearly, any lead line $\gamma$ which determines $\p$ satisfies $[\min \p, \max \p] \subset (\gamma)_1$.
A straightforward corollary of Lemma \ref{lem:updownhairs} reads as follows.

\begin{lemma}\label{lem:updownprimeends}
Let $\p, \q \in \L(U)$ be accessible prime ends and let $p, q$ be their principal points, respectively. Suppose $\p \prec \q$.
\begin{itemize}
\item If $(q)_1 < \min \p$ then $[w(\p, \q)] = -1$.
\item If $\max \q < (p)_1$ then $[w(\p, \q)] = 1$.
\end{itemize}
\end{lemma}

\section{Accessible periodic points}

The goal of this section is to familiarize the reader with the tools introduced in the previous sections and show how can they be used
in practice. We will prove that the rotation number of any accessible periodic point must be equal to the prime end rotation number.
As was mentioned in the introduction, other different approaches to this result can be found in \cite{cartwright}, \cite{yorke} and \cite{duke}.

\begin{theorem}\label{thm:accessibleperiodic}
Suppose $f \colon \A \to \A$ is an orientation--preserving homeomorphism, $X$ is a non--empty closed invariant set such that $\A \setminus X$
is homeomorphic to $\A$ and $F$ and $\X$ are lifts of $f$ and $X$, respectively.
Let $q$ be a lift of a periodic point in $X$ and let $a, b$ be integers such that $F^b(q) = T^a(q)$.
Then $\widehat{\rho}(F, X) = a/b$.
\end{theorem}
\begin{proof}
Replace $F$ by $T^{-a}F^b$ so $q$ becomes a fixed point of $F$, it has rotation number equal to 0.
Argue by contradiction and assume $\rho = \widehat{\rho}(F, X) < 0$ (the positive case can be addressed by replacing $F$ by $F^{-1}$).

Denote $\q$ an accessible prime end whose principal point is $q$ and $\gamma_{\q}$ a lead line that determines $\q$.
Take a positive integer $m$ large enough so that $\max(\gamma_{\q})_1 < (T^m(q))_1$. Since $\rho < 0$, we can find
a positive integer $n$ such that $\F^n \T^m(\q) \prec \q$.
Notice that the principal point of $\F^n \T^m (\q)$ is $F^nT^m(q) = T^mF^n(q) = T^m(q)$.
Lemma \ref{lem:updownprimeends} yields $[w(\F^n\T^m(\q), \q)] = 1$ and since both principal points lie in the same horizontal line
we actually have $w(\F^nT^m(\q), \q) = 1$.

Now, we will compute the previous winding number using another approach and we will get a different outcome.
The new element in the picture is $\F^n(\q)$, an accessible prime end that satisfies
$$
\F^n(\q) \prec \F^n\T^m(\q) \prec \q.
$$
Consider a lead line $\gamma_0$ determining $\F^n(\q)$ such that $T^m(\gamma_0) \cap \gamma_0 = \emptyset$.
Such a lead line can be obtained simply by lifting a suitable lead line in the annulus.
Denote $\gamma_1 = T^m(\gamma_0)$. Since $\F^n(\q) \prec \F^n\T^m(\q)$, we have that $(\gamma_0(0))_1 < (\gamma_1(0))_1$.
The map $\phi_0\colon [0,1] \to S^1$ defined by
\[
\phi_0(t) = \frac{\gamma_1(t) - \gamma_0(t)}{||\gamma_1(t) - \gamma_0(t)||}
\]
\noindent is constant and so is any of its lifts.
Therefore, $w(\F^n(\q), \F^n\T^m(\q)) = w(\gamma_0, \gamma_1) = 0$.
Define
\begin{align*}
\phi_1(t) =
\begin{dcases}
\frac{\gamma_1(0) - \gamma_0(2t)}{||\gamma_1(0) - \gamma_0(2t)||} & 0 \le t \le 1/2 \\
\frac{\gamma_1(2t-1) - q}{||\gamma_1(2t-1) - q||} & 1/2 \le t \le 1.\\
\end{dcases}
\end{align*}
The map $\phi_1$ is homotopic to $\phi_0$ relative to $\{0,1\}$ so the lift $\widetilde{\phi_1}$ of $\phi_1$
with $\widetilde{\phi_1}(0) = 0$ satisfies $\widetilde{\phi_1}(1) = 0$.

Let $\gamma_2$ be a lead line that determines $\q$ and is disjoint to $\gamma_0, \gamma_1$ except for its endpoint $q = \gamma_2(1) = \gamma_0(1)$.
It satisfies $(\gamma_1(0))_1 < (\gamma_2(0))_1$.
Define
\begin{align*}
\phi_2(t) =
\begin{dcases}
\frac{\gamma_1(0) - \gamma_2(2t)}{||\gamma_1(0) - \gamma_2(2t)||} & 0 \le t \le 1/2 \\
\frac{\gamma_1(2t-1) - q}{||\gamma_1(2t-1) - q||} & 1/2 \le t \le 1.\\
\end{dcases}
\end{align*}
Since multiplication by a constant ($-1$ in this case) does not affect to the computation of the winding number,
we deduce that $w(\F^n\T^m(\q), \q) = w(\gamma_1, \gamma_2) = \widetilde{\phi_2}(1)$, where $\widetilde{\phi_2}$ is the lift of $\phi_2$ that vanishes at 0.

Given two maps $\psi, \psi'\colon [0,1] \to S^1$, $\psi(1) = \psi'(0)$, define the reverse $\psi^{R}$ and the concatenation $\psi\star\psi'$
by $\psi^{R}(t) = \psi(1-t)$ and $\psi \star \psi'(t) = \psi(2t)$ if $0 \le t \le 1/2$ and
$\psi \star \psi'(t) = \psi'(2t-1)$ if $1/2 \le t \le 1$.
It is easy to show that the lifts to $\R$ of these maps that vanish at 0 satisfy: $\widetilde{\psi^{R}}(1) = - \widetilde{\psi}(1)$ and
$\widetilde{\psi \star \psi'}(1) = \widetilde{\psi}(1) + \widetilde{\psi'\phantom{\,}}(1)$.

\begin{figure}[htb]
\begin{center}
\includegraphics[scale = 0.8]{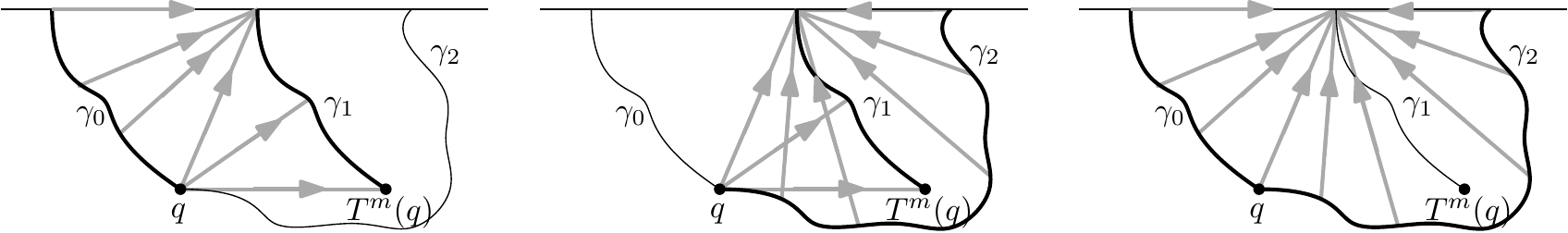}
\end{center}
\caption{Gray arrows display maps $\phi_1$ (left), $\phi_2$ (middle), $\phi_4$ (right).}
\label{fig:periodicpoints}
\end{figure}

Consider the lift to $\R$ of the concatenation of $\phi_1$ and $\phi_2^{R}$ that vanishes at 0, say $\widetilde{\phi_3}$.
We have that $\widetilde{\phi_3}(1) = \widetilde{\phi_1}(1) - \widetilde{\phi_2}(1)$.
The map $\phi_3$ is homotopic, relative to $\{0,1\}$, to a map $\phi_4$
defined by

\centerline{
\noindent $\phi_4(t) = \displaystyle \frac{\gamma_1(0) - \gamma_0(2t)}{||\gamma_1(0) - \gamma_0(2t)||}$ if $t \in [0,1/2]$
and $\phi_4(t) = \displaystyle \frac{\gamma_1(0) - \gamma_2(2-2t)}{||\gamma_1(0) - \gamma_2(2-2t)||}$ if $t \in [1/2,1]$.
}

\noindent
The lift of $\phi_4$ that vanishes at 0, $\widetilde{\phi_4}$, satisfies $\widetilde{\phi_4}(1) = 1$.
Therefore,

\smallskip

\centerline{
$1 = \widetilde{\phi_4}(1) = \widetilde{\phi_3}(1) = \widetilde{\phi_1}(1) - \widetilde{\phi_2}(1)= -\widetilde{\phi_2}(1)$, and $w(\F^n\T^m(\q), \q) = \widetilde{\phi_2}(1) = -1$.
}

\smallskip
\noindent This contradicts our first computation.

\end{proof}

\section{Rotational horseshoe}\label{sec:horseshoe}

Before we pass to the proof of the results presented in this article, let us describe a classical example in
the study of rotation in surface dynamics: the rotational horseshoe.
This example shows that an accessible point $q$ with bounded forward orbit may not have well--defined rotation number,
i.e. $\lim_{n \to +\infty} (F^n(q))_1/n$ may not exist. At the same time, it is an example of how accessible points with different rotation numbers can coexist.

Consider an orientation--preserving homeomorphism of $\A$ whose action is illustrated in Figure \ref{fig:horseshoe}.
The square $D$ is first squashed in the vertical direction, then stretched into a long strip, $f(D)$, that meets $D$ in two rectangles.
The restriction of $f$ to $D$ is equal to the classical horseshoe map. The set $\cap_{n \le 0} f^n(D)$ is a Cantor set of vertical lines
$\{V_a: a \in \{0,1\}^{\N}\}$ such that $f(V_a) \subset V_{\sigma(a)}$, where $\sigma\colon \{0,1\}^{\N} \to \{0,1\}^{\N}$
is the shift map.

\begin{figure}[htb]
\begin{center}
\includegraphics[scale = 0.55]{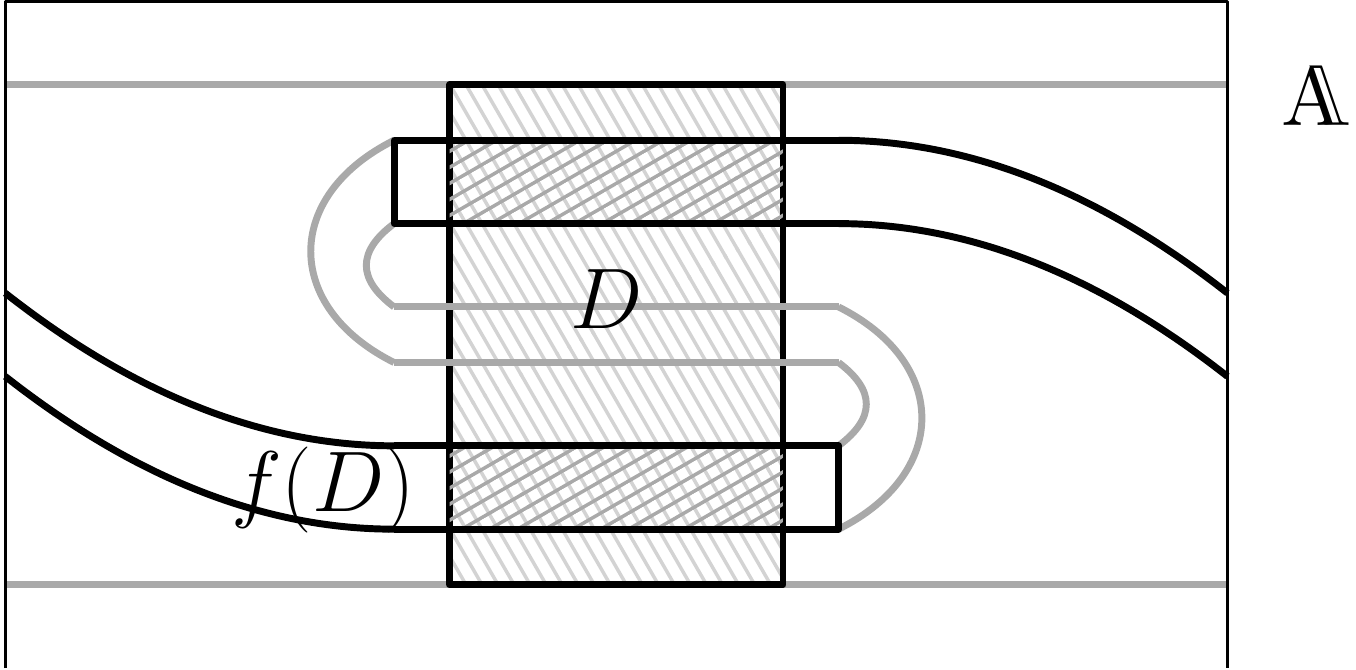}
\end{center}
\caption{Left and right sides are identified. $D$ is shadowed and $D \cap f(D)$ is doubly shadowed.}
\label{fig:horseshoe}
\end{figure}

The map $f$ has a lift $F \colon \AA \to \AA$ defined by imposing that $F(\widetilde{D})$ meets $\widetilde{D}$
and $T(\widetilde{D})$, where $\widetilde{D}$ denotes a fixed lift of $D$. The rotational behavior of a point of a vertical line $V_a$ is
easily expressed in terms of its codification $a \in \{0,1\}^{\N}$:

If $x$ belongs to the lift of $V_a$ contained in $\widetilde{D}$ then $F^n(x)$ lies in $T^{s_n}(\widetilde{D})$,
where $s_n = a_0 + \ldots + a_{n-1}$ and $a = (a_0,a_1, \ldots, a_n, \ldots)$. As a consequence,
the difference between $|(F^n(x))_1 - (x)_1|$ and $a_0 + \ldots + a_{n-1}$ is bounded by 1 in absolute value and
$$
\liminf_{n \to +\infty} \frac{(F^n(x))_1}{n} = \liminf_{n \to +\infty} \frac{a_0 + \ldots a_{n-1}}{n} \enskip \enskip \text{and}
\enskip \enskip \limsup_{n \to +\infty} \frac{(F^n(x))_1}{n} = \limsup_{n \to +\infty} \frac{a_0 + \ldots a_{n-1}}{n}.
$$

\noindent
It is easy to construct a sequence $b = (b_0, b_1, \ldots) \in \{0,1\}^{\N}$ for which
$$\liminf_{n \to +\infty} \frac{b_0 + \ldots b_{n-1}}{n} \neq \limsup_{n \to +\infty} \frac{b_0 + \ldots b_{n-1}}{n}.$$
Then, the previous discussion shows that for any point $x$ in the lift of $V_b$
$$
\lim_{n \to +\infty} \frac{(F^n(x))_1}{n} \enskip \text{does not exist,}
$$
i.e. those points do not have a well--defined (forward) rotation number.
In a similar fashion, by choosing a sequence in $\{0,1\}^{\N}$ in which the density of 1's is equal to $\lambda$ we obtain points whose (forward) rotation number is equal to any prescribed real number $\lambda$ in $[0,1]$.

The rotational horseshoe described above can be made part of a dynamics as the ones considered in this work.
The union of $D$ and the rectangular region with gray border in Figure \ref{fig:horseshoe} is an attracting closed annulus $A$, $f(A) \subset A$.
The limit $K = \cap_{n \ge 0} f^n(A)$ is an essential annular continuum that splits $\A$ in two regions.
Define $X$ as the union of $K$ and the lower region, its complement $\A \setminus X$ contains $\partial \A$ and is homeomorphic to $\A$.
Clearly, every vertical line $V_a$ contains at least one accessible point of $X$. For the previous choice
of $b \in \{0,1\}^{\N}$, no accessible point in the lift of $V_b$ has a well--defined forward rotation number under $F$.
Other clever choices of binary sequence provide accessible points with any prescribed forward rotation number in $[0,1]$.

\section{Transverse drift}\label{sec:transverse}

Recall that a semi--orbit $\O^+(p)$ or $\O^-(p)$ in $\AA$ is called bounded if the projection to $\A$ is relatively compact.
The following statement corresponds to Theorem \ref{thm:resumen} (\ref{item:1}) in the introduction.

\begin{theorem}\label{thm:diffdirections}
There do not exist accessible points $p, q \in \X$ such that $\O^-(p)$ and $\O^+(q)$ are bounded and
\begin{align*}
\lim_{n \to -\infty} (F^n(p))_1 - n \rho = -\infty \enskip \text{and} \enskip
\lim_{n \to +\infty} (F^n(q))_1 - n \rho = +\infty,\\
\text{or} \enskip\lim_{n \to -\infty} (F^n(p))_1 - n \rho = +\infty \enskip \text{and} \enskip
\lim_{n \to +\infty} (F^n(q))_1 - n \rho = -\infty,
\end{align*}
where $\rho = \widehat{\rho}(F, X)$.
\end{theorem}

\smallskip
The homotopy between $f$ and the identity map lifts to a homotopy $\{h_s\}_{s \in [0,1]}$.
It follows that the displacement of a point $x$ under the homotopy can be bounded
$$|(x)_1 - (h_s(x))_1| < D(x), \enskip \forall s \in [0,1],$$
by a map $D\colon \AA \to \R^+$ which is invariant under translation: $D(T(x)) = D(x)$.
It is important to remark that $D$ has a well--defined maximum in lifts of compact sets of the annulus
and, as a consequence, the restriction of $D$ to bounded orbits or semi--orbits is bounded.


As an application of Lemma \ref{lem:endhomotopy} we have:

\begin{proposition}\label{prop:toofar}
Let $\e, \e'$ be accessible prime ends in $\X$ and $e, e'$ their principal points, respectively.
If $|(e)_1 - (e')_1| > D(e) + D(e')$ then
$$[w(\e, \e')] = [w(\F(\e), \F(\e'))].$$
\end{proposition}

\begin{proof}
Take $\gamma, \gamma'$ disjoint lead lines that determine $\e, \e'$, respectively.
For any $s \in [0,1]$ we have,
$$|(h_s(\gamma(1)))_1 - (h_s(\gamma'(1)))_1| \ge |(e)_1 - (e')_1| - |(e)_1 - (h_s(\gamma(1)))_1| - |(h_s(\gamma'(1)))_1 - (e')_1|.$$
By definition of $D$, the last two terms are bounded by $D(e)$ and $D(e')$, respectively, and we deduce that
$(h_s(\gamma(1)))_1 \neq (h_s(\gamma'(1)))_1$. Lemma \ref{lem:endhomotopy} yields the result.
\end{proof}

\begin{proof}[ of Theorem \ref{thm:diffdirections}]
Let us proceed by contradiction.
Assume $p, q$ as in the first alternative of the statement exist,
$$\lim_{n \to -\infty} (F^n(p))_1 - n \rho = -\infty \enskip \text{and} \enskip
\lim_{n \to +\infty} (F^n(q))_1 - n \rho = +\infty.$$
The second alternative follows from the first one by replacing $f$ by $f^{-1}$.
Let $\p, \q$ be accessible prime ends whose principal points are $p$ and $q$, respectively.

Choose $\c$ an accessible prime end of $\widetilde{U}$ as a point of reference in $\L(U)$.
Choose $\p_{n}, \q_n$ prime ends obtained from $\F^{n}(\p), \F^n(\q)$ by translation under $\T$
so that
\[\T^{-1} \c \preceq \q_{n} \prec \c \prec \T \c \preceq \p_n \prec \T^2 \c,
\]
\noindent and denote by $p_{n}, q_n$ their principal points, respectively.
Assume without loss of generality that $\p_0 = \p, \q_0 = \q$.

\medskip
\noindent
\emph{Claim:} $\lim_{n \to -\infty} (p_{n})_1 = - \infty, \enskip \lim_{n \to + \infty} (q_n)_1 = + \infty$.

\smallskip

Indeed, if $\q_n = \T^{-m_n} \F^n(\q)$ then the numbers $m_n - n \rho$
are bounded by 1. 
Since

\smallskip
\centerline{$(q_n)_1 + (m_n - n \rho) = (F^n(q))_1 - n \rho$,}

\noindent it follows that $\lim_{n \to +\infty}(q_n)_1 = +\infty$.
We can argue similarly for $p_{n}$.
\medskip

Set $C = \sup\{D(x): x \in \mathcal{O}^-(p) \cup \mathcal{O}^+(q)\} = \sup\{D(x): x \in \{p_{n}\}_{n \le 0} \cup \{q_n\}_{n \ge 0}\}$.
It is finite because both semi--orbits are bounded.

Choose first $k_0$ so that $(q_k)_1 > C + 1$ for every $k \ge k_0$.
By Lemma \ref{lem:updownprimeends}, if $(p_{n})_1 < \min \q_{k_0}$ then $[w(\q_{k_0}, \p_{n})] = -1$.
Take $n_0$ such that for every $n \le n_0$ the previous condition on the relative winding number is satisfied
and also $(p_n)_1 < -C - 1$.
If we now pick $k_1$ so that $(q_{k_1})_1 > 1 + \max \p_{n_0}$ we deduce also from Lemma \ref{lem:updownprimeends}
that $[w(\q_{k_1}, \T^j\p_{n_0})] = 1$ for every $j \in \{-1,0,1\}$. The configuration is sketched in Figure \ref{fig:transversedrift}.

\begin{figure}[htb]
\begin{center}
\includegraphics[scale = 0.9]{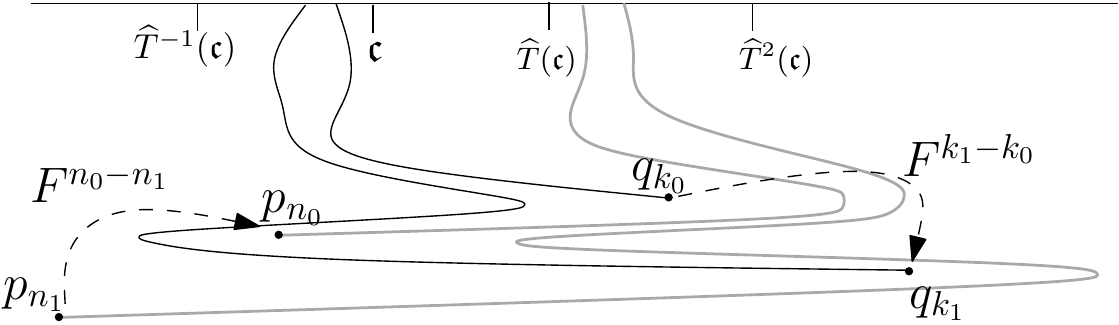}
\end{center}
\caption{Configuration of lead lines that determine $\q_{k_0}, \q_{k_1}$ (black) and $\p_{n_0}, \p_{n_1}$ (gray).}
\label{fig:transversedrift}
\end{figure}

To conclude the proof, set $n_1 = n_0 - (k_1 - k_0)$.
For any $0 \le i \le n_0-n_1 = k_1-k_0$, if $q_{k_0+i} = T^m F^i(q_{k_0})$
and $p_{n_1+i} = T^{m'} F^i(p_{n_1})$ then $|m - m'| \le |m - i\rho| + |m' - i\rho| \le 2$.
As a consequence,
$$(F^i(q_{k_0}))_1 - (F^i(p_{n_1}))_1 \ge (q_{k_0+i})_1 - (p_{n_1+i})_1 - 2 > 2C \ge D(F^i(q_{k_0})) + D(F^i(p_{n_1}))$$
 and we can apply Proposition \ref{prop:toofar} to obtain
$$[w(\F^i(\q_{k_0}),\F^i(\p_{n_1}))] = [w(\F^{i+1}(\q_{k_0}),\F^{i+1}(\p_{n_1}))].$$
Thus, $[w(\q_{k_0}, \p_{n_1})] = [w(\F^{k_1-k_0}\q_{k_0}, \F^{n_0-n_1} \p_{n_1})] = [w(\q_{k_1}, \T^j(\p_{n_0}))]$
for some $j \in \{-1,0,1\}$, which is absurd.
\end{proof}

\begin{remark}
The speed of convergence in the limits of Theorem \ref{thm:diffdirections} does not play a role.
\end{remark}

An immediate application of Theorem \ref{thm:diffdirections} yields this interesting corollary.

\begin{corollary}\label{cor:accessiblemustrotate}
If the full orbit of a point $x \in X$ is bounded and has well--defined rotation number, that is,
$\rho_x = \lim_{n \to \pm \infty} \left(F^n(\tilde{x})\right)_1 / n \in \R$
for a lift $\tilde{x}$ of $x$ and $\rho_x \neq \widehat{\rho}(F, X)$ then $x$ is not an accessible point of $X$.
\end{corollary}

Theorem \ref{thm:diffdirections} is not true if we replace $\lim$ by $\liminf$ or $\limsup$ in the statement.

\begin{theorem}\label{thm:exampletransversal}
Under our standing hypothesis, given any $\rho_0 \in \R$ we can construct an example of dynamics such that
$\widehat{\rho}(F, X) = \rho_0$ and there are accessible points $p, q \in \X$ with $\O^-(p)$, $\O^+(q)$ bounded and satisfying
\begin{align*}
\liminf_{n \to -\infty} (F^n(p))_1 -n\rho_0 = \liminf_{n \to +\infty} (F^n(q))_1 -n\rho_0 = -\infty &
\enskip\enskip\enskip \text{and}\enskip \\
\limsup_{n \to -\infty} (F^n(p))_1 -n\rho_0 = \limsup_{n \to +\infty} (F^n(q))_1 -n\rho_0 = +\infty & \\
\end{align*}
\end{theorem}
\begin{proof}
We make the construction for $\widehat{\rho}(F, X) = 0$.
It can be adapted to an arbitrary prime end rotation number without much effort.

Fix $\theta_0 \neq \theta_1 \in S^1$ and consider
a homeomorphism $h\colon \A \to \A$ given by $h(\theta, r) = (\theta, \varphi(\theta, r))$ and such that
$\varphi(\theta_0, r) < r < \varphi(\theta_1, r)$ for every $r < 0$. The forward (resp. backward) semi--orbit of $x_0 = (\theta_0, -1)$
(resp. $x_1 = (\theta_1, -1)$) under $h$ tends to $e^-$, the lower end of $\A$.
We need the motion under $h$ to be slow close to $e^-$, it suffices to ask for $|\varphi(r,\theta) - r| \to 0$ as $r \to -\infty$.

Consider $\phi\colon \A \to \A$ given by the change of coordinate $\theta \mapsto \theta + r \sin(2\pi r)$.
Denote $g = \phi \circ h \circ \phi^{-1}$ and let $l(\theta) = \{(\theta, r): r \le 0\}$ the vertical line in $\A$.
Choose $G\colon \AA \to \AA$ a lift of $g$ that fixes pointwise $\partial \AA$.
The lift of $\phi(l(\theta_i))$, $i=0,1$, to the universal cover $\AA$ is a line
that oscillates arbitrarily in $\widetilde{\theta}$, the lift of the angular coordinate (see Figure \ref{fig:ejemploeses})
and so does the semi--orbit of a lift of $\phi(x_i)$ under $G$.
Thanks to the properties of $\varphi$, if $\iota_r\colon S^1 \times \{r\} \to \A$ denotes the inclusion map,
$||\restr{g}{S^1 \times \{r\}} - \iota_r|| \to 0$ as $r \to -\infty$.

Now, we plug $\psi\colon\A \to S^1 \times (-2, -1]$ using the change of coordinates $r \mapsto 1/(r^2+1)-2$.
The convergence of $\restr{g}{S^1 \times \{r\}}$ guarantees that we can extend the conjugate dynamics $\psi^{-1} \circ g \circ \psi$
to a homeomorphism $f\colon \A \to \A$ which is the identity outside $S^1 \times (-2, -1]$.
The set

\smallskip
\centerline{$X = (\psi \circ \phi)(l(\theta_0) \cup l(\theta_1)) \cup (S^1 \times (-\infty, -2])$}
\smallskip

\noindent is closed and invariant under $f$.

Let $F$ be the lift of $f$ that is equal to $\widetilde{\psi}^{-1}\circ G \circ \widetilde{\psi}$ in $S^1 \times [-2,-1]$,
where $\widetilde{\psi}$ is any lift of $\psi$.
The forward (resp. backward) semi--orbit of a lift $q$ of $\psi \circ \phi (x_0)$
(resp. a lift $p$ of $\psi \circ \phi (x_1)$) under $F$ satisfies
\begin{align*}
\liminf_{n \to +\infty} \,(F^n(q))_1 = -\infty & \enskip \text{and} \enskip
\limsup_{n \to +\infty} \,(F^n(q))_1 = +\infty.\\
\text{(resp. }
\liminf_{n \to -\infty} \,(F^n(p))_1 = -\infty & \enskip \text{and} \enskip
\limsup_{n \to -\infty} \,(F^n(p))_1 = +\infty
\text{ ).}
\end{align*}
Trivially, the semi--orbits $\O^-(p)$ and $\O^+(q)$ under $F$ are bounded.

\begin{figure}[htb]
\begin{center}
\includegraphics[scale = 0.7]{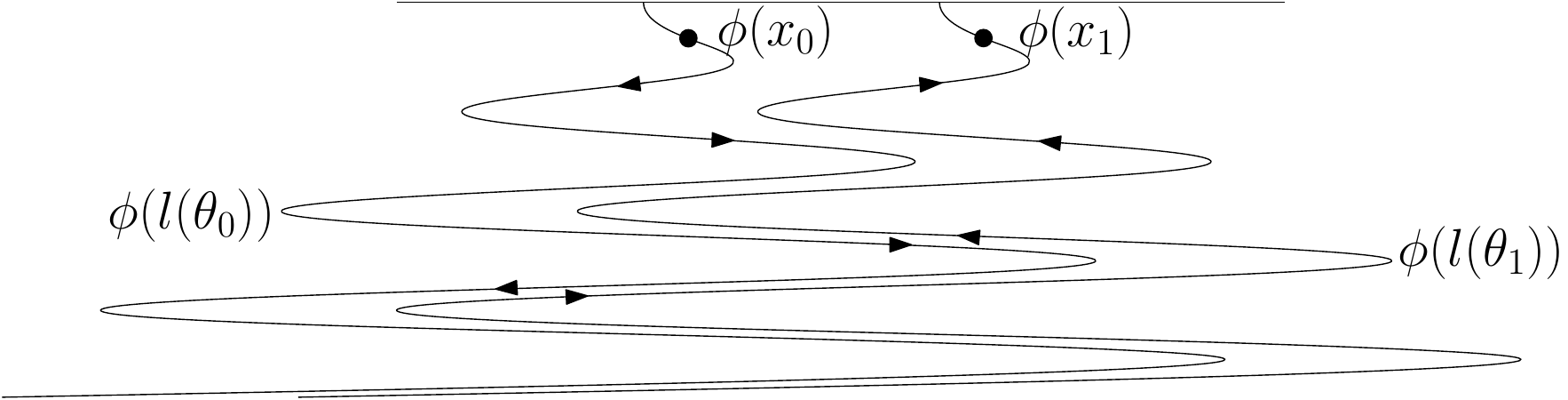}
\end{center}
\caption{Lift of the lines $\phi(l(\theta_0))$ and $\phi(l(\theta_1))$ and the dynamics generated by $G$.}
\label{fig:ejemploeses}
\end{figure}

Since $F$ fixes pointwise $\R \times \{-1\}$, $F$ fixes every cross--cut contained in it. This means that there are prime ends fixed by $\widehat{F}$ in $\L(U)$ and it follows that $\widehat{\rho}(F,X) = 0$.
\end{proof}

\subsection{Invariant measures}

Every Borel probability measure on $\A$ which is invariant under the action of $f$
has an associated rotation number:
$$\rho(F, \mu) = \int_{\A} u \; d\mu,$$
where $u \colon \A \to \R$ is the displacement function defined by $u(x) = (F(\tilde{x}))_1 - (\tilde{x})_1$,
$\tilde{x}$ being any lift of $x$.

Invariant measures can always be found in compact $\omega$--limits of orbits (hence in compact invariant sets),
for example as weak limits of measures of the form $\sum_{k = 0}^n \delta_{f^k(x)}/n$.
Invariant measures can be written as a combination of ergodic invariant measures, that is, measures which assign value 0 or 1
to every invariant set.

\begin{theorem}
Let $\mu$ be a Borel probability ergodic $f$--invariant measure which is supported in a compact subset of $X$.
Suppose that $\rho(F, \mu) \neq \widehat{\rho}(F, X)$. Then the set of accessible points of $X$ has zero $\mu$--measure.
\end{theorem}
\begin{proof}
Applying Birkhoff's Ergodic Theorem to $F$ and $F^{-1}$,
\begin{align*}
\rho(F, \mu) = \int_{\A} u \; d \mu = \lim_{n \to +\infty} \frac{1}{n} \sum_{k = 0}^{n-1} u \circ F^k(x) \enskip \text{and} \enskip
\rho(F, \mu) = \int_{\A} u \; d \mu = \lim_{n \to +\infty} \frac{1}{n} \sum_{k = 0}^{n-1} u \circ F^{-k}(x)\\
\end{align*}
for $\mu$--almost every $x$.
The limits on the right hand sides of the equations are equal to $\lim_{n \to +\infty} (F^n(x))_1 / n$
and $\lim_{n \to -\infty} (F^n(x))_1 / n$, respectively.
This implies that the full orbit of $\mu$--almost every point in $X$ has a rotation number equal to $\rho(F, \mu)$.
However, recall that by Corollary \ref{cor:accessiblemustrotate}
an accessible point cannot have a well--defined rotation number different from $\widehat{\rho}(F, X)$.
\end{proof}

\section{Unbounded drifts}\label{sec:mainresults}

Assume there are accessible points $p, q \in \X$ such that $\O^-(p)$ and $\O^+(q)$ are bounded
and these semi--orbits drift from the expected one.
In other words,
$\{(F^n(p))_1 - n \rho\}_{n \le 0}$ and $\{(F^n(q))_1 - n \rho\}_{n \ge 0}$ are unbounded for $\rho = \widehat{\rho}(F,X)$.
We already know from Theorem \ref{thm:exampletransversal}
that both sequences can have subsequences that diverge to $\infty$ (in any direction).
Nevertheless, as was proved in Theorem \ref{thm:diffdirections},
the sequences themselves cannot tend to $+\infty$ and to $-\infty$, respectively, or viceversa.
We shall see now that this qualitative obstruction disappears if the sequences diverge towards the same end of $\R$.

\begin{theorem}\label{thm:counterexamplebouncing}
Under our standing hypothesis, given any $\rho_0 \in \R$ there is an example of dynamics such that
 $\widehat{\rho}(F, X) = \rho_0$ and there is an accessible point $p$ in $\X$ with bounded orbit such that
\begin{align*}
\lim_{n \to -\infty} (F^n(p))_1 - n \rho_0 = + \infty \enskip \enskip \text{and} \enskip \enskip
\lim_{n \to +\infty} (F^n(p))_1 - n \rho_0 = + \infty.
\end{align*}
\end{theorem}
\begin{proof}
Before we give the details let us describe the dynamics that will be constructed.
The invariant set $X$ will consist of several strings which spiral towards a limit circle together with the
subannulus bounded by this circle, as in Figure \ref{fig:primeends}.
The dynamics in the limit circle will be a rotation of angle $\rho_0$ in both cases.
The set of strings and the dynamics induced on it corresponds to the closure of a wandering orbit of a circle homeomorphism with
rotation number $\rho_0 + \Z$ and the restriction of the circle dynamics to it.
A small push in the radial coordinate will force the orbits of points in the strings to spiral towards or outwards the limit circle.
In the universal cover, the dynamics acts in the strings as if they were infinite escalators whose speed decreases near the end and
the direction of movement, up or down, depends on whether the string belongs to the forward or backward orbit of a fixed string.

Take a homotopy of homeomorphisms $\tau_r \colon S^1 \to S^1$, $r \in (-\infty, 0]$, all of which have
rotation number equal to $\rho_0 + \Z$ and converge uniformly as $r \to -\infty$ to the rigid rotation of angle $\rho_0 + \Z$, $\tau_{-\infty}$.
The homotopy also satisfies that there is a continuous family

\smallskip
\centerline{$\{\theta_{r,0} \in S^1: r \in (-\infty, 0]\}$}

\smallskip
\noindent
such that $\{\theta_{r,n} = \tau^n_r(\theta_{r,0}): n \in \Z\}$ is a wandering orbit of $\tau_r$ and
for every $n \in \Z$ the limit of $\theta_{r,n}$ as $r \to -\infty$ exists. Denote the previous limit by $\theta_{-\infty, n}$.
The previous properties imply that $\theta_{-\infty, n} = \tau_{-\infty}^n(\theta_{-\infty, 0})$.
Perhaps it is useful to think of the family $\{\tau_r\}$ as a bifurcation of the rotation $\tau_{-\infty}$.

Consider a homeomorphism $h \colon \A \to \A$ that satisfies $h(\theta, r) = (\tau_{r}(\theta), r)$ if $\theta$ belongs to
$\Lambda_r$, the set of accumulation points of $\{(\theta_{r,n},r): n \in \Z\}$,
and sends the points $(\theta_{r, n}, r)$ to $(\theta_{r', n+1}, r')$, where
$r' = r + \frac{r}{n(r^2+1)}$ if $n \neq 0$ and $r' = r$ if $n = 0$.
Both conditions are compatible because $n \to \infty$ forces
$r'$ to tend to $r$ and $\theta_{r,n}$ to approach $\Lambda_r$. Moreover, we impose that the oscillation of the radial coordinate under $h$
reaches its maximum at $\theta_{r, \pm 1}$ and, in particular, it tends to 0 as $r \to -\infty$.

\medskip
\emph{Claim:} Both the forward and backward orbit of $x_0 = (\theta_{-1,0}, -1)$ under $h$ tend to $e^-$,
the lower end of the annulus.

Denote $r_n$ the radial coordinate of $h^n(x_0)$, $r_0 = -1$. By definition of $h$, $r_{n+1} = r_n + \frac{r_n}{n(r_n^2+1)}$ if $n \neq 0$
so $r_n$ decreases as $n \to \pm\infty$.
Suppose that the forward orbit of $x_0$ does not tend to $e^-$.
Then, the absolute value of its radial coordinate remains bounded by a constant $C \ge 1$,
$-r_n \le C$.
Since $r_n \le -1$, for $n \ge 1$ we have
$$r_{n+1} - r_n = \frac{r_n}{n(r_n^2+1)} \le - \frac{C}{n(C^2+1)} < -\frac{1}{n(C+1)}.$$
 By induction, we deduce that for every $n \ge 1$
$$r_{n+1} - r_1 \le - \sum_{k = 1}^n \frac{1}{k(C+1)}$$
which contradicts the boundedness of $\{r_n\}_{n \ge 0}$.
The proof of the claim for the backward orbit is analogous.
\medskip

We now twist $\A$ using $\phi\colon \A \to \A$ defined by $\theta \mapsto \theta - r$.
The image of every vertical line $l_n = \{(\theta_{r,n}, r): r \le 0\}$ by $\phi$
is wrapped around $\A$ infinitely many times in the positive direction.
The last assertion follows from the fact that for a fixed $n$ the points $\theta_{r, n}$
have a well--defined limit as $r$ goes to $-\infty$ equal to $\theta_{0,n}$ .
Denote $g=\phi\circ h \circ \phi^{-1}$ and let $\widetilde{\tau}_r\colon \R \to \R$ be the lift of $\tau_r$ with rotation number $\rho_0$
and $G$ be the lift of $g$ to the universal cover $\AA$ which coincides with $\widetilde{\tau}_0$ in $\partial \AA = \R \times \{0\}$.
Since any lift $\widetilde{l}_n$ of $\phi(l_n)$ to $\AA$ is tilted with respect to $l_n$
and $g^n(\phi(x_0)) \xrightarrow[n \to \pm \infty]{} e^-$,

\[(G^n(\widetilde{y_0}))_1 - \widetilde{\tau}_{-1}^n(\widetilde{\theta}_{-1,0}) \xrightarrow[n \to \pm\infty]{} +\infty,\]

\noindent where $\widetilde{y_0}$ is a lift of $\phi(x_0)$. See Figure \ref{fig:ejemploboomerang}.
Notice that $|\widetilde{\tau}_{-1}^n(\widetilde{\theta}_{-1,0}) - n \rho_0|$ is bounded by a small constant independent of $n$, so

$$(G^n(\widetilde{y_0}))_1 - n \rho_0 \xrightarrow[n \to \pm\infty]{} +\infty,$$

\begin{figure}[htb]
\begin{center}
\includegraphics[scale = 0.7]{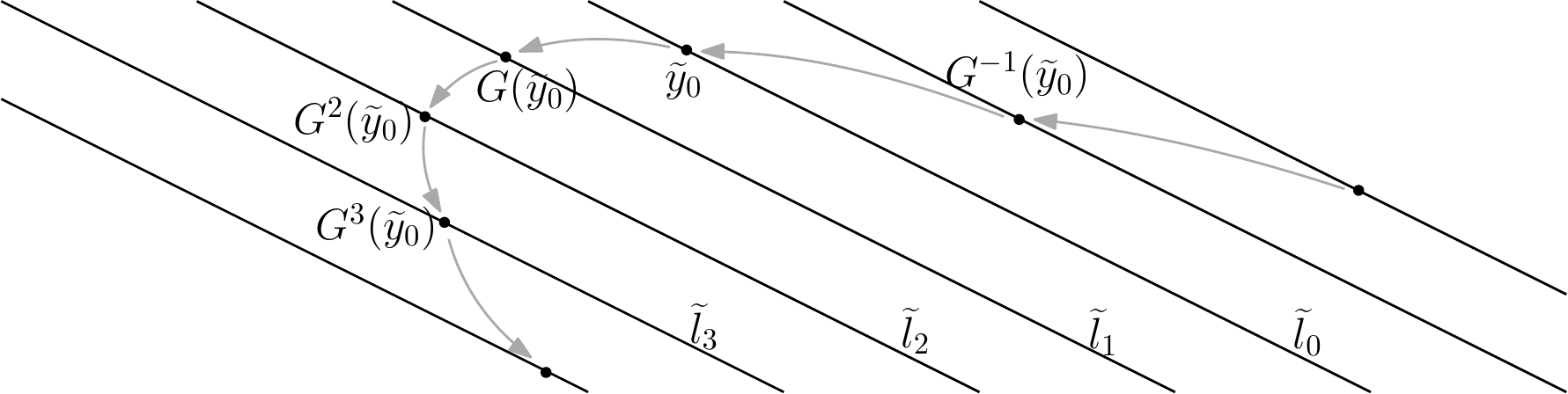}
\end{center}
\caption{Orbit of $\widetilde{y_0}$ under $G$.}
\label{fig:ejemploboomerang}
\end{figure}

In the slices $S^1 \times \{r\}$ of $\A$, $||\restr{g}{S^1 \times \{r\}} - \iota_r \tau_{-\infty} \iota_r^{-1}|| \to 0$ as $r \to -\infty$.
Equivalently,

\smallskip
\centerline{$||g(\theta, r) - (\theta + \rho_0, r)||$ tends to 0 as $r$ goes to $-\infty$ uniformly in $\theta \in S^1$.}
\smallskip

\noindent
In words, this means the action of $g$ in the circles $S^1 \times \{r\}$ becomes the rotation by $\rho_0$ as we approach the lower end of the annulus.
Indeed, if $(\theta', r') = g(\theta, r)$ we have that $|r' - r| \le \frac{r}{r^2+1} \to 0$ as $r \to -\infty$.
The angle satisfies $\theta' = \tau_r(\theta + r) - r'$. From the previous convergence and the fact that $\tau_r \to \tau_{-\infty}$ we deduce
\begin{align*}
\theta' - \theta &= \tau_r(\theta + r) - r' - \theta = (\tau_r(\theta + r) - (\theta + r + \rho_0)) + (r - r') + \rho_0 = \\
 &= (\tau_r(\theta + r) - \tau_{-\infty}(\theta + r)) + (r - r') + \rho_0 \xrightarrow[r \to -\infty]{} \rho_0.\\
\end{align*}

Now, we proceed as in Theorem \ref{thm:exampletransversal} and use $\psi$ to define a homeomorphism in $S^1 \times (-2,1]$.
Extend it by $(\theta, r) \mapsto (\tau_0(\theta), r)$
in $S^1 \times [-1, 0]$ and by $(\theta, r) \mapsto (\tau_{-\infty}(\theta), r)$ in $S^1 \times (-\infty, 2]$
to obtain a map $f$ in $\A$. According to the preceding paragraph, $f$ is a global homeomorphism.
Then, $X = (\psi \circ \phi)(\cup_{n\in \Z} l_n) \cup \{(\theta, r): \theta \in \Lambda_r\} \cup (S^1 \times (-\infty, -2])$
and $p = \psi (\widetilde{y_0})$ satisfy the conditions in the statement.

The dynamics in the set of twisted lines has not been transformed throughout the construction. For every $\widetilde{\theta}$ in the lift $\widetilde{\Lambda_0}$ of $\Lambda_0$ there is a vertical lead line $\gamma_{\widetilde{\theta}}$ that lands at $(\widetilde{\theta}, -1) \in \X$. Denote $\p_{\widetilde{\theta}}$ the prime end determined by $\gamma_{\widetilde{\theta}}$ and $\mathcal{P} = \{\p_{\widetilde{\theta}}: \widetilde{\theta} \in \widetilde{\Lambda_0}\}$. The order $\prec$ in $\mathcal{P}$ corresponds to the natural order of $\widetilde{\Lambda_0}$ as a subset of $\R$. Furthermore, $F(\gamma_{\widetilde{\theta}}) = \gamma_{\widetilde{\tau}_0(\widetilde{\theta})}$ so $\F(\p_{\widetilde{\theta}}) = \p_{\widetilde{\tau}_0(\widetilde{\theta})}$. As a consequence $\widehat{\rho}(F, X) = \rho(\widetilde{\tau}_0) = \rho_0$.
\end{proof}

\begin{remark}
The rational case in the previous theorem can be also addressed with a simpler construction.
We may start with an angular dynamics equal to the rigid rotation of angle $\rho_0+\Z$ in every level $-\infty < r \le 0$ and take
the strings that correspond to a fixed periodic orbit. Then, we can thicken this finite set of strings to a set of bars
which after the twist look like infinite bars placed parallel to the direction $\theta+r = 0$ in $\A$.
Maintaining the boundary of the set of bars invariant (as a set), it is possible to perturb the map so that
the dynamics on the boundaries looks, up to permutation of the bars, like the movement of an infinite chainsaw whose speed decreases to 0
as we approach the lower end of $\A$.
\end{remark}

Theorem \ref{thm:counterexamplebouncing} shows that the qualitative obstruction that appears in the transverse drift case
and is expressed in Theorem \ref{thm:diffdirections} does not generalize to any kind of unbounded drifts.
The previous construction has a special feature: the orbits of the accessible points exhibiting an unbounded drift from the
expected rotational behavior converge to a limit circle whose dynamics is a rigid rotation of angle $\rho_0$.
For example, if the prime end rotation number is 0 this means that any accumulation point of the drifting orbits is fixed.
In particular, the speed of the drift decreases to 0 even though the drift is unbounded.

The goal of this section is to prove that introducing mild conditions on the rate of divergence of
$\{(F^n(q))_1 - n \rho\}_{n \ge 0}$ (for example growing at least linearly on $n$) for a single accessible point $q$ of $\X$
with $\O^+(q)$ bounded implies that all bounded backward orbits of accessible points have rotation number equal to $\rho$.
This result is formally stated in Theorem \ref{thm:main} and together with
Theorem \ref{thm:counterexamplebouncing} show that the speed of convergence in the limits,
in spite of going unnoticed in the transverse drift case, does play a role in general.

To ease the exposition, we suppose from now on that $\limsup_{n \to +\infty} \, (F^n(q))_1 - n \rho = + \infty$,
where $\rho = \widehat{\rho}(F, X)$ and $\O^+(q)$ is bounded.
The case in which $\{(F^n(q))_1 - n \rho\}_{n \ge 0}$ is unbounded from below is basically equal, a few words to address it
are included before the end of the section.

Henceforth, we also assume the following two mild conditions on the drifting forward semi--orbit of $q$ are satisfied:

\smallskip
\noindent
There exist integers $a, b$ such that $b \ge 1$,
\begin{itemize}
\item[(H1)] $\rho \le a/b =: \rho'$ and $\limsup_{n \to +\infty} (F^n(q))_1 - n \rho' = + \infty$ and
\item[(H2)] there are strictly increasing sequences of integers $\{m_i\}_{i=1}^{\infty}, \{n_i\}_{i=1}^{\infty}$ and an integer $k$, $0 \le k < b$, such that
$T^{-m_i - an_i} F^{bn_i + k}(q)$ converges as $i \to +\infty$ to a point $q_0 \in \X$ that is not fixed under $T^{-a} F^b$.
\end{itemize}

\medskip
A trivial remark is that (H1) prevents $q$ from being a periodic point by Theorem \ref{thm:accessibleperiodic}.

\begin{lemma}\label{lem:equalasymptotics}
For any $b \in \N, k \in \Z$ the condition
$$\limsup_{n \to + \infty} (F^{bn+k}(x))_1 - b n \rho = + \infty \enskip\enskip
\text{is equivalent to} \enskip \enskip\limsup_{n \to + \infty} (F^{n}(x))_1 - n \rho = + \infty$$
$$\text{and} \enskip \enskip \limsup_{n \to + \infty} (F^{bn+k}(x))_1/n > \rho \enskip\enskip
\text{is equivalent to} \enskip \enskip\limsup_{n \to + \infty} (F^{n}(x))_1/n > \rho$$
provided $\O^+(x)$ is bounded.
The same equivalence holds if the limits are taken as $n \to -\infty$ and $\O^-(x)$ is bounded.
\end{lemma}
\begin{proof}
We only prove the first equivalence in the case $n \to +\infty$, the rest of the cases can be proven in a similar fashion.
For any $m \in \Z$, denote by $n$ the unique integer such that $m \in [bn + k, b(n+1) + k)$. We have that
\begin{align*}
\big| ((F^{bn+k}(x))_1 - bn\rho) - ((F^m(x))_1 - m \rho) \big| & \le \big| (F^{bn+k}(x))_1 - (F^m(x))_1\big| + |(m-bn)\rho| <   \\
 & < b \, \max\{D(y): y \in \mathcal{O}^{+}(x))\} + (b + |k|)\,|\rho| =: C.\\
\end{align*}
Thus, for any subsequence of $\{(F^m(x))_1 - m \rho\}_{m \ge 1}$ we find a subsequence of \[\{(F^{bn+k}(x))_1 - b n \rho\}_{n \ge 1}\]
such that the difference between corresponding terms is bounded by the constant $C$ and the result follows.
\end{proof}

Let us prove that the hypothesis (H1)--(H2) are weaker than assuming that the drift of the forward orbit of $q$ grows at least linearly.

\begin{lemma}\label{lem:linealdriftH1H2}
If $\displaystyle\limsup_{n \to +\infty} \frac{(F^n(q))_1}{n} > \rho$ then (H1) and (H2) hold.
\end{lemma}
\begin{proof}
Take $\rho' = a/b \in \Q$ satisfying $\limsup_{n \to +\infty} (F^n(q))_1/n > \rho' \ge \rho$.
Then
$$\limsup_{n \to +\infty} (F^n(q))_1/n - \rho' > 0\enskip\enskip \text{so} \enskip\enskip
\limsup_{n \to +\infty}(F^n(q))_1 - n\rho' = +\infty \enskip\enskip \text{as in (H1)}.$$

Now, we prove (H2).
By the choice of $\rho'$ and Lemma \ref{lem:equalasymptotics}, the condition in the statement implies
$\limsup_{n \to +\infty}(F^{bn}(q))_1/n > b\rho' = a$.
Consequently, denoting $G = T^{-a}F^b$, there exists $\delta > 0$ such that
$(G^n(q))_1 = (F^{bn}(q))_1 - an> \delta n$ holds for infinitely many $n$ and, in particular,
$\limsup_{n \to +\infty} (G^n(q))_1 = +\infty$.
Therefore, we can find an increasing integer sequence $\{n_i\}_{i=1}^{\infty}$
such that
$$(G^{n_i+1}(q))_1 > (G^{n_i}(q))_1 + \delta/2 \;(\star)\enskip \enskip \text{and} \enskip \enskip
(G^{n_{i+1}}(q))_1 > (G^{n_i}(q))_1 + 1 \;(\star\star)$$
for every $i \ge 1$.
After passing to a subsequence, we can suppose by $(\star\star)$ that there is an increasing integer sequence $\{m_i\}_{i=1}^{\infty}$ such that
$T^{-m_i}G^{n_i}(q)$ converges to $q_0 \in \X$. Finally, $(\star)$ implies
$(G(q_0))_1 > (q_0)_1 + \delta/2$ and $q_0$ is not fixed by $G = T^{-a}F^b$.
\end{proof}

Let us briefly discuss the hypothesis (H1)--(H2) in terms of $\rho$.
In the case $\rho$ is irrational, if (H1) is satisfied then $\rho' > \rho$ so the drift is at least linear, that is,
$\limsup_{n \to + \infty} (F^n(q))_1/n > \rho$. According to Lemma \ref{lem:linealdriftH1H2}, (H2) is automatically satisfied as well.
The same conclusion holds as long as (H1) is valid for some $\rho' > \rho$.
In the case $\rho$ is rational and (H1) only holds for $\rho' = \rho$, i.e. the drift is sublinear, hypothesis (H2) discard dynamics
as the ones constructed in Theorem \ref{thm:counterexamplebouncing} with a limit cycle composed exclusively of periodic points.
In fact, if (H1) holds with $\limsup$ replaced by $\lim$ then then (H2) is always satisfied unless every point
of accumulation of the orbit of the projection $q$ to $\A$ is periodic and has rotation number $\rho'$.
This is the content of the following lemma.

\begin{lemma}\label{lem:allperiodic}
Suppose $\lim_{n \to +\infty} (F^n(q))_1 - n \rho' = +\infty$ for some rational $\rho' = a/b \ge \rho$
and the forward orbit of the projection of $q$ to $\A$ under $f$ has an accumulation point that is not
a periodic point of rotation number $a/b$. Then (H1) and (H2) hold.
\end{lemma}
\begin{proof}
Since (H1) is trivially satisfied, we pass directly to (H2).
Theorem \ref{thm:accessibleperiodic} implies $q$ is not periodic.
Denote $q_0$ a lift of the accumulation point provided by the statement
and denote by $q_i$ lift of points of the forward orbit of the projection of $q$ such that $q_i \to q_0$ as $i \to +\infty$.
By hypothesis $F^b(q_0) \neq T^a(q_0)$.
The points $q_i$ have the form $q_i = T^{-r_i} F^{s_i}(q)$ for some integers $r_i, s_i$ such that $\{s_i\}$ is increasing.
After passing to a subsequence we can suppose $\{s_i\}$ is constant modulo $b$ equal to $k$, $0\le k < b$, that is,
$s_i = b s'_i + k$.
The limit in the statement implies
$$\lim_{i \to +\infty} (F^{bs'_i+k}(q))_1 - as'_i =
 \lim_{i \to +\infty} (F^{bs'_i+k}(q))_1 - (bs'_i+k)\rho' + k\rho' = +\infty.$$
Thus, since the points $q_i$ remain in a compact set, $\{(F^{bs'_i+k}(q))_1 - r_i\}_{i}$ is bounded and
 the sequence $r_i - as'_i \to +\infty$ as $i \to +\infty$.
Again, after passing to a subsequence we can suppose that $r_i - as'_i$ is increasing in $i$.
Then, (H2) is satisfied for the sequences $\{m_i = r_i - as'_i\}, \{n_i = s'_i\}$ and the integer $k$:
$$T^{-m_i-an_i}F^{bn_i+k}(q) = T^{-r_i}F^{bs'_i+k}(q) = q_i \xrightarrow[i\to +\infty]{ } q_0.$$
\end{proof}

We are ready to state the main theorem of this section.

\begin{theorem}\label{thm:main}
Let $f\colon \A \to \A$ be an orientation--preserving homeomorphism, $X$ a non--empty invariant closed set such that $\A \setminus X$
is homeomorphic to $\A$ and $F\colon \AA \to \AA$ and $\X$ lifts of $f$ and $X$, respectively, to the universal cover of $\A$.
Suppose there exists an accessible point $q \in \X$ with bounded forward semi--orbit and such that
$$
\limsup_{n \to +\infty} \, (F^n(q))_1 - n \rho = + \infty,
$$
where $\rho = \widehat{\rho}(F,X)$, and (H1) and (H2) are satisfied.
Then, for any accessible point $p \in \X$ with bounded backward semi--orbit there exists $C > 0$ such that
$$\big|(F^n(p))_1 -(p)_1 - n \rho \big| < C$$
for any $n \le 0$.
\end{theorem}

\begin{proof}
First, we show that it is enough to prove the result in the case $\rho' = 0$ (from (H1)) and $k = 0$ (from (H2)).
Indeed, take $a, b, k$ satisfying (H1)--(H2) and write $\rho' = a/b$ and $q' = F^k(q)$.
According to Lemma \ref{lem:equalasymptotics}
$$\limsup_{n \to +\infty} (F^n(q))_1 - n \rho' = + \infty \enskip \text{is equivalent to} \enskip \limsup_{n \to +\infty} (G^n(q))_1 = + \infty,$$
where $G = T^{-a}F^b$, which is in turn equivalent to
$\limsup_{n \to +\infty} (G^n(q'))_1 = + \infty$.
Property (H2) supplies $q_0$, the limit of $T^{-m_i-an_i}F^{bn_i+k}(q) = T^{-m_i}G^{n_i}(q')$, that is not fixed by $G$.
For any point $p$,
\[
\{(G^n(p))_1 - (p)_1 - n (\rho - \rho')\}_{n \le 0} \enskip \text{is unbounded if and only if so is} \enskip
\{(F^n(p))_1 - (p)_1 - n \rho\}_{n \le 0},
\]
again by Lemma \ref{lem:equalasymptotics}.
Evidently, $\X$ is invariant under $G$, $G(\X) = T^{-a}F^b(\X) = T^{-a}(\X) = \X$
and the semi--orbit of any point under $G$ is bounded if and only if the semi--orbit under $F$ is also bounded.
In sum, after replacing $F$ by $G = T^{-a}F^b$ and $q$ by $q'$ we can assume $\rho' = k = 0$.
Since $\widehat{\rho}(G, X) = b \widehat{\rho}(F, X) - a \le 0$, we henceforth assume $\rho \le 0$ as well.

After the reduction discussed in the previous paragraph,
the set of hypothesis on the dynamics and the forward orbit of $q$ that will be used in the proof are that
\begin{itemize}
\item $\rho \le 0$,
\item (H1): $\limsup_{n \to + \infty} (F^n(q))_1 = + \infty$ and
\item (H2): there exists increasing sequences of integers $\{m_i\}_{i=1}^{\infty}, \{n_i\}_{i=1}^{\infty}$ such that\\
$T^{-m_i}F^{n_i}(q) \xrightarrow[i \to +\infty]{} q_0 \in \X$ and $q_0$ is not fixed by $F$.
\end{itemize}

Since $q_0$ is not fixed by $F$,
let us take a closed topological disk $B$ in $\AA$ that contains $q_0$ in its interior and is free, that is, $F(B) \cap B = \emptyset$.
Assume further that $B$ does not touch the boundary of $\AA$. $B$ will play a major role in the proof.
Without loss of generality, we suppose that the points $T^{-m_i}F^{n_i}(q)$ belong to $B$.
Note that all these points are different. Indeed, otherwise $q$ would be periodic and (H1) would contradict Theorem \ref{thm:accessibleperiodic}.

Let us give a brief yet inaccurate sketch of the proof of Theorem \ref{thm:main} to shed some light in the forthcoming propositions.
Denote $\q$ an accessible prime end whose principal point is $q$ and define $\q_i = \T^{-m_i}\F^{n_i}(\q)$,
where $m_i, n_i$ come from (H2).
Since $\rho \le 0$ and $\{m_i\}, \{n_i\}$ are increasing, the sequence of accessible prime ends $\{\q_i\}$
is strictly decreasing and tends to $-\infty$ in $\L(U)$.
Notice, however, that their principal points $q_i = T^{-m_i}F^{n_i}(q)$ lie in $B$.
According to Lemma \ref{lem:goodleadlinessequence}, there exist lead lines $\gamma_i$ that determine $\q_i$ and are pairwise disjoint.

The right question to raise now concerns the location of the image of $B$, $F(B)$. We only know that it does not meet $B$.
Then, there are basically two possibilities for $F(B)$. The first one states that $F(B)$ is ``in front of'' $B$ in the sense that
an infinite number of lead lines $\gamma_i$ encounter $F(B)$ before $B$.
Intuitively, this means that $B$ is deeper than $F(B)$ as we head for the corners and cavities of $\widetilde{U} = \AA \setminus \X$.
However, this situation ultimately implies that $\{(F^n(q))_1\}_{n \ge 0}$ is bounded from above and this contradicts our hypothesis,
see Proposition \ref{prop:notinfront}. The second possibility
says roughly that $F(B)$ is placed ``behind'' $B$ or deeper than $B$.
More precisely, the lead lines that land at $F(B)$ first go through $B$.
The conclusion in this case, see Proposition \ref{prop:haciaatrasrotanigual},
is that the deviation of every bounded backward orbit from the expected rotation is bounded by a constant.
In particular, all backward rotation numbers are equal to $\rho$.

The disk $B$ defined above is central in the proof and leads to the following definitions.
The subset of accessible prime ends of $\widetilde{U}$ whose principal point is contained in $B$
is denoted $\acc_B$. The choice of $B$ implies that $\acc_B$ is unbounded from below.
The set of accessible prime ends $\p$ such that every lead line that determines $\p$ meets $B$ is
denoted $\acc$. Trivially, $\acc_B \subset \acc$.

One of the technical issues we have to deal with throughout this proof concerns the relative position
of the lead lines $\gamma_i$ determining $\q_i$ and $B$. Since our argument does not use the fact that
the points $q_i$ converge to $q_0$, we will replace the prime ends $\q_i$ by another sequence of prime ends
contained in $\acc_B$ that also diverges to $-\infty$ in $\L(U)$ and for which we can find lead lines suitably placed.
In particular, these lead lines will be pairwise disjoint and their intersection with $B$ will be closed subarcs containing
the landing points.
The precise requirements are stated in the following lemma.

\begin{lemma}\label{lem:goodleadlines}
Given a neighborhood $W$ of $B$, $\p \in \acc$ and a lead line $\gamma$ that determines $\p$,
there exists $\p' \in \acc_B$ with $\p' \preceq \p$ and a lead line $\gamma'$ that determines $\p'$ such that
\begin{itemize}
\item $\gamma' \subset \gamma \cup W$ and
\item $\gamma' \cap B = \gamma' \cap \partial B$ is a singleton or a final closed subarc of $\gamma'$. 
\end{itemize}
\end{lemma}
\begin{proof}
Let us give a first guess on the construction of the lead line $\gamma'$. We can follow $\gamma$ until it first meets $B$ and then continue along an arc of $\partial B$ until we reach $\X$. This naive method yields a lead line $\gamma'$ but does not guarantee that $\gamma'$ determines a prime end $\p'$ smaller or equal than $\p$ as Figure \ref{fig:lemat} illustrates.

Denote $\mathcal{C}$ the set of crosscuts of $\widetilde{U}$ determined by $\partial B$.
Trivially, $\partial B \cap \X$ is not empty because $B$ contains points of $\X$ and $\X$ has no compact connected component.
In the case $\partial B \cap \X$ is a single point, $\mathcal{C}$ is empty (recall that the endpoints of a crosscut must be different)
but it will be clear from the ensuing arguments how to construct $\gamma$.
Therefore, let us assume from now on that $\mathcal{C}$ is a non--empty set.
Since the elements of $\mathcal{C}$ are non--trivial
closed arcs in $\partial B$ whose interiors are pairwise disjoint, there are at most countably many crosscuts in $\mathcal{C}$.
Note also that if $\{c_n\}_n$ is an infinite sequence of different elements of $\mathcal{C}$, then any of its limit points,
i.e. points $x$ that can be written as $x = \lim_i x_{n_i}$ with $x_{n_i} \in c_{n_i}$ and $n_i \to +\infty$, belongs to $\X$.

{\centering
\begin{figure}[htb]
\begin{tabular}{l c r}
\includegraphics[scale = 0.5]{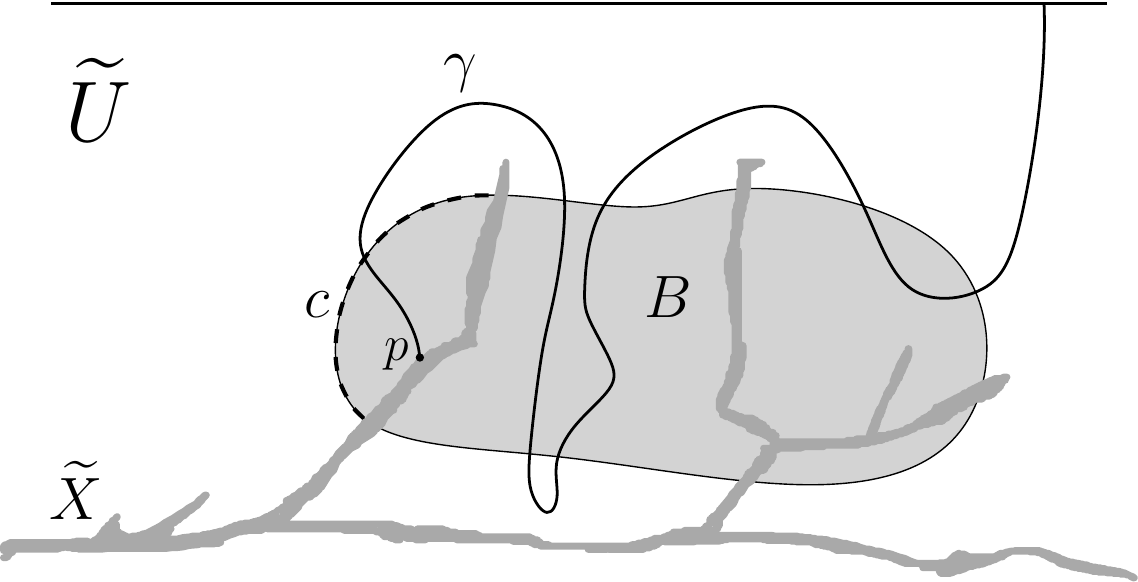} & \hspace{5mm} & \includegraphics[scale = 0.5]{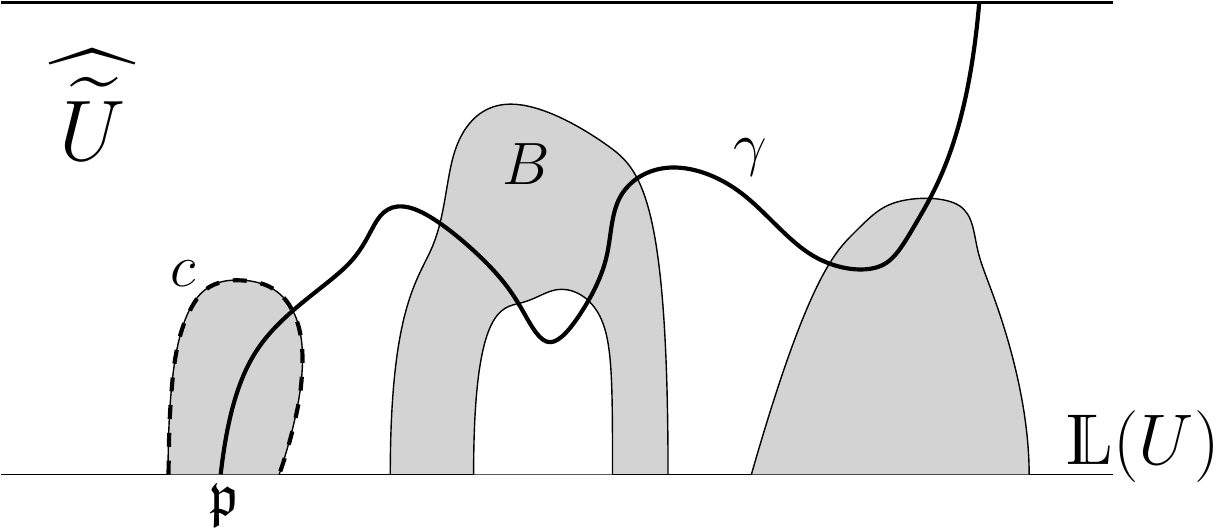}
\end{tabular}
\caption{In this example, before traversing the crosscut $c$ (marked in dashed line), $\gamma$ intersects (twice) three other crosscuts contained in $\partial B$. Those three crosscuts determine prime ends in $\L(U)$ that are greater than $\p$ and two of them are maximal.}
\label{fig:lemat}
\end{figure}
}

There is a natural order relation in $\mathcal{C}$: $c_i$ is larger than $c_j$ iff $V(c_j) \subset V(c_i)$ (tildes are dropped from the original notation of Subsection \ref{subsec:lineend}).
Notice that for any pair of crosscuts in $\mathcal{C}$ there are three possibilities: $V(c_i) \subset V(c_j), V(c_j) \subset V(c_i)$
and $V(c_i) \cap V(c_j) = \emptyset$.
Denote $\mathcal{C}^*$ the subset of $\mathcal{C}$ composed of the crosscuts which are maximal with respect to the previous order.
As a consequence, if $c_i, c_j$ are different crosscuts in $\mathcal{C}^*$ then $V(c_i) \cap V(c_j) = \emptyset$.
The subset of $\mathcal{C}^*$ composed of the maximal crosscuts $c$ such that $\p$ is not adherent to $V(c)$
in $\widehat{\phantom{\,}\widetilde{U}\phantom{\,}}$ is denoted $\mathcal{C}_0^*$.
Since every crosscut of $\mathcal{C}$ is a non--trivial closed sub--arc of $\partial B$,
the set $\mathcal{C}_0^* = \{c_j^0\}_{j \in J}$ is countable.
For every $j \in J$, $\p$ is not adherent to $V(c_j)$ so $\gamma \cap c_j$ is a compact subset of the open arc $\interior(c_j)$.
Take a closed interval $\alpha_j$ in the
interior of $c_j$ that contains $\gamma \cap c_j$. Since no point of $\alpha_j$ is accumulated by the rest of the crosscuts
in $\mathcal{C}_0^*$ and $V(c_j)$ is homeomorphic to an open disk, we can find pairwise disjoint closed topological disks $D_j$, $j \in J$,
that are also disjoint to the crosscuts in $\mathcal{C}^* \setminus \mathcal{C}_0^*$, and contain in their interior
both $\alpha_j$ and $\gamma \cap V(c_j)$. We further suppose that $D_j \cap \X = \emptyset$, $D_j \subset V(c_j) \cup W$ and the diameter of $D_j$
tends to 0 in the case $J$ is infinite.
Now, we can push $\gamma$ outside every $V(c_j)$ using homeomorphisms $h_j\colon D_j \to D_j$ such that $(h_j)_{|\partial D_j} = \id$ and
$h_j(\gamma \cap D_j) \cap  V(c_j) = \emptyset$. This implies $h_j(\gamma \cap D_j) \subset W \setminus B$.
Construct a homeomorphism $h$ supported in $\cup_j D_j$ by $h_{|D_j} = h_j$ and the identity in the rest.
The effect of $h$ is to push $\gamma$ outside any of the unsuitable crosscuts it meets: there is no crosscut $c \in \mathcal{C}$
such that $\p$ is not adherent to $V(c)$ and $h(\gamma)$ intersects $c$, i.e., $h(\gamma)$ does not meet any
element of $\mathcal{C}_0^*$.
It is easy to see that the lead line $h(\gamma)$ still determines $\p$.
Write $\gamma_0 = h(\gamma)$.

Let $t_0$ be the smallest parameter at which $\gamma_0$ meets $B$, $t_0 = \min\{t \in [0,1]: \gamma_0(t) \in B\}$.
There are now two possibilities for the lead line $\gamma_0$: either $t_0 < 1$ or $t_0 = 1$.
The latter alternative implies that $\gamma_0$ is disjoint to $B$ except for its landing point so the statement is already satisfied if we take $\p' = \p$
and $\gamma' = \gamma_0$.
Otherwise, $\gamma_0(t_0)$ does not belong to $\X$ so it lies in a crosscut in $\partial B$, say $c$.
The crosscut $c$ has to be maximal in the sense explained above and does not belong to $\mathcal{C}_0^*$.
Therefore, $c$ determines two different prime ends and
one of them is not greater than $\p$, say $\p'$. Let $\gamma'$ be the concatenation of the arc $\gamma_0[0,t_0]$ and the arc in $c$ that
starts at $\gamma_0(t_0)$ and determines $\p'$. Then, $\gamma'$ is a lead line that determines $\p'$ and
satisfies the statement.
\end{proof}

\begin{lemma}\label{lem:goodleadlinessequence}
Let $\{\e_i\}_{i = 0}^{\infty}$ be a sequence of accessible prime ends contained in $\acc_B$ that diverges to $-\infty$ in $\L(U)$,
and a set of lead lines $\gamma_i$ that determine $\e_i$.
Then, given any neighborhood $W$ of $B$ there exists a sequence of accessible prime ends $\{\e'_k\}_{k = 0}^{\infty}$
and lead lines $\gamma'_k$ satisfying:
\begin{enumerate}
\item The sequence $\{\e'_k\}$ is contained in $\acc_B$, it is strictly decreasing and tends to $-\infty$ in $\L(U)$.
\item $\gamma'_k$ determines $\e'_k$, $\gamma'_k \subset \gamma_k \cup W$ and $\gamma'_k \cap B = \gamma'_k \cap \partial B$ is either a single point or a final closed subarc of $\gamma'_k$.
\item $\{\gamma'_k\}$ are pairwise disjoint except for the possible coincidence of landing points and $\gamma'_k(0)$ tends to $-\infty$ in $\partial \AA$.
\end{enumerate}
\end{lemma}
\begin{proof}
The proof goes by induction, applying Lemma \ref{lem:goodleadlines}
at every step and doing the necessary modifications to fulfill (iii).
Firstly, apply Lemma \ref{lem:goodleadlines} to $\e_0, \gamma_0$ and $W$ to obtain $\e'_0, \gamma'_0$ satisfying (ii).
Now, we prove the inductive step. Assume we have constructed $\e'_0, \ldots, \e'_k$ and $\gamma'_0, \ldots, \gamma'_k$ satisfying
the statement. Take $\e_i \prec \e'_k$ and apply Lemma \ref{lem:goodleadlines} to $\e_i, \gamma_i$ and $W$
to obtain $\e'_{k+1}$ and $\gamma'$. We have that $\e'_{k+1} \prec \e_i \prec \e'_k$ and that $\gamma'$ satisfies (ii).
However, we still have to do some modification to address (iii).


If $\gamma'$ is disjoint to $\cup_{j = 0}^k \gamma'_j$, write $\gamma'_{k+1} = \gamma'$ and proceed directly to the final paragraph of the proof.
Otherwise, let $t'$ be the maximum value among all $0 \le t < 1$ for which $\gamma'(t) \in \cup_{j=0}^k \gamma'_j$.
Since $\e'_{k+1} \prec \e'_k$ the maximum is well--defined and $t' < 1$.
Lemma \ref{lem:primeends2} guarantees $\gamma'(t')$ belongs to $\gamma'_k$, $\gamma'(t') = \gamma'_k(t_k)$.

Consider a small neighborhood $V$ of $\gamma_k[0,t_k]$ disjoint to $B$ and $\X$ and $\epsilon > 0$ such that
$\gamma'(t'+\epsilon) \in V$. Then, we can replace $\gamma'[0,t' + \epsilon]$ by an arc
that goes parallel to $\gamma'_k$ within $V$ without ever touching it and ending at $\gamma'(t' + \epsilon)$.
This new lead line $\gamma'_{k+1}$ also determines $\e'_{k+1}$ and is disjoint to $\gamma'_k$
and so to every $\gamma'_j$, $j = 0,\ldots,k$.

Finally, modify $\gamma'_{k+1}$ by pulling it left in a neighborhood of $\partial \AA$ that is disjoint to $B$
so that $\gamma'_{k+1}(0) < \gamma'_k(0) - 1$ in $\partial \AA$ and no intersections are created.
This last adjustment guarantees that $\gamma'_{k}(0) \xrightarrow[k \to +\infty]{ } -\infty$.
\end{proof}


\begin{lemma}\label{lem:boundeddomain}
Let $\gamma_0, \gamma_1$ be disjoint lead lines that determine the prime ends $\e_0, \e_1$, respectively, with $\e_0 \prec \e_1$. Suppose $\alpha\colon [0,1] \to \AA$ is an arc that joins the principal point of $\e_0$ with the principal point of $\e_1$ and $\alpha \cap \partial \AA = \emptyset$. Assume further that there are $0 \le t_0 < t_1 \le 1$ such that $\alpha[0,t_0]$ coincides with a final part of $\gamma_0$, $\alpha[t_1, 1]$ is the final piece of $\gamma_1$ and $\alpha(t_0,t_1)$ does not touch $\gamma_0 \cup \gamma_1$. Denote $D$ the domain enclosed by $\gamma_0, \alpha(t_0,t_1), \gamma_1$ and $\partial \AA$. For every accessible $\e \in [\e_0, \e_1]$ there are two possibilities:
\begin{itemize}
\item either there is a lead line $\gamma$ contained in $\overline{D}$ that determines $\e$
\item or every lead line that determines $\e$ intersects $\alpha$.
\end{itemize}
As a consequence, there exists $C > 0$ such that if $e$ is the landing point of $\e$, $\e \in [\e_0, \e_1]$ and $|(e)_1| > C$ then
every lead line that determines $\e$ intersects $\alpha$.
\end{lemma}
\begin{proof}
Suppose $\gamma'$ is a lead line that determines $\e$ and $\gamma'\cap\alpha=\emptyset$.
This implies that $\e \neq \e_0, \e_1$ because the landing points of $\e_0, \e_1$ belong to $\alpha$.
If $\gamma'$ is contained in $\overline{D}$ set $\gamma = \gamma'$ and skip the rest of this paragraph.
Otherwise, $\gamma'$ meets $\gamma_0$ or $\gamma_1$. Define $t' = \max\{t: \gamma'(t) \in \gamma_0 \cup \gamma_1\}$ and assume for instance that $\gamma'(t') = \gamma_0(t_0)$. Note that $t', t_0 < 1$ because the landing point of $\gamma_0$ belongs to $\alpha$.
Replace $\gamma'[0,t']$ by the arc $\gamma_0[0,t_0]$ and call the resulting lead line $\gamma$.
Assume $\gamma$ is parameterized in a way that $\gamma(t_0) = \gamma_0(t_0)$.
Clearly, $\gamma$ also determines $\e$ and is disjoint to $\alpha$.

The construction of $\gamma$ ensures that if $\gamma$ meets $\interior(D)$ then $\gamma \subset \overline{D}$.
In the prime end closure of $\widetilde{U}$, $\gamma_0$ determines two closed regions
and $\gamma$ is contained in the same region as $\gamma_1$ because of the inequality $\e_0 \prec \e \prec \e_1$.
Since $\gamma \cap \gamma_1 = \emptyset$, we conclude that $\gamma_0(0) \le \gamma(0) < \gamma_1(0)$.
In the case $\gamma(0) \in (\gamma_0(0), \gamma_1(0))$ it is clear that $\gamma(\epsilon) \in \interior(D)$ for small $\epsilon > 0$ and
the conclusion follows.
Otherwise, let us show that $\gamma(t_0 + \epsilon) \in \interior(D)$ if $\epsilon > 0$ is small.
As in the previous lemma, in a small neighborhood of $\gamma_0[0,t_0]$
construct a hanging arc $\gamma_{\epsilon}$ that goes parallel to $\gamma$ without ever touching $\gamma_0 \cup \X \cup \alpha \cup \gamma_1$
and lands at $\gamma(t_0 + \epsilon)$.
Since $\gamma_{\epsilon}$ can be extended to a lead line that determines $\e$ and is disjoint to $\gamma_0$,
we see that $\gamma_{\epsilon}(0) \in (\gamma_0(0), \gamma_1(0))$
and this yields that $\gamma_{\epsilon}(0,1] \subset \interior(D)$. In particular, $\gamma_{\epsilon}(1) = \gamma(t_0 + \epsilon)$ lies in $\interior(D)$ and, as a consequence, $\gamma \subset \overline{D}$.

The last assertion in the statement follows from the compactness of $\overline{D}$ in $\AA$.
\end{proof}

\begin{remark}\label{rmk:boundeddomain}
Lemma \ref{lem:boundeddomain} can be easily rephrased to address the case in which the prime ends $\e_0, \e_1$ share a common principal point $p$. Let us state it for later use. Let $\gamma_0, \gamma_1$ lead lines that determine $\e_0, \e_1$, respectively, and are disjoint except for their common landing point. Denote $D$ the domain enclosed by $\gamma_0, \gamma_1$ and $\partial \AA$. Then, for every $\e \in [\e_0, \e_1]$ there is a lead line $\gamma$ contained in $\overline{D}$ that determines $\e$. In particular, there is a constant $C > 0$ such that if $e$ is the principal point of $\e \in [\e_0, \e_1]$ then $|(e)_1| \le C$.

The proof of this statement only requires a small modification of the previous argument.
\end{remark}

The previous lemmata supply all the technical details necessary in the proof of Theorem \ref{thm:main}.
Recall that the proof is basically divided in two propositions dealing with the relative position of $B$ and $F(B)$.
Intuitively, the first proposition states that $F(B)$ lies deeper than $B$ in the sense
that there do not exist a sequence of accessible prime ends having their principal points in $B$, diverging to $-\infty$ and whose images
under $\F$ can be accessed without going through $B$.

\begin{proposition}\label{prop:notinfront}
$\acc_B \cap \widehat{F}^{-1}(\L(U) \setminus \acc)$ is bounded from below.
\end{proposition}
\begin{proof}
Let us proceed by contradiction.
Assume there is a sequence of accessible prime ends $\{\e_k\}_{k = 0}^{\infty}$, in $\acc_B$ converging to $-\infty$ such that
$\widehat{F}(\e_k) \notin \acc$. Then, we can find lead lines $\gamma_k$ that determine $\e_k$ and $F(\gamma_k) \cap B = \emptyset$.
Take a neighborhood $W$ of $B$ so that $F(W) \cap B = \emptyset$.
Apply Lemma \ref{lem:goodleadlinessequence} to $\{\e_k\}, \{\gamma_k\}$ and $W$ to obtain a new sequence of accessible prime ends $\{\e'_k\}$ and pairwise disjoint lead lines $\{\gamma'_k\}$. We deduce that $F(\gamma'_k) \cap B  \subset (F(\gamma_k) \cup F(W)) \cap B = \emptyset$.
Henceforth, denote $\gamma''_k = F(\gamma'_k)$.

The lead lines $\gamma''_k$ split $\AA \setminus F(B)$ in many connected components.
There are countably many of them that meet $\partial \AA$,
say $\{U_k\}_{k = 0}^{\infty}$. They can be described as follows.
Each component $U_k, k \ge 1,$ is enclosed by $\gamma_{k-1}$, $\gamma_k$, $\partial F(B)$ and $\partial \AA$,
hence it is relatively compact in $\AA$. Recall from Lemma \ref{lem:goodleadlinessequence} that $\gamma_k(0) \to -\infty$,
so the complement of $\cup_{k = 1}^{\infty}\overline{U_k} \cap \partial \AA$ in $\partial \AA$ is a half--line
converging to the $+\infty$ end that is contained in $U_0$.

\medskip

\noindent
\emph{Claim:} $(\L(U) \setminus \acc) \cap \widehat{F}^{-1}(\acc)$ is bounded from below.
\smallskip

\noindent
Suppose we have an accessible prime end such that $\e \in \L(U) \setminus \acc$, $\F(\e) \in \acc$ and $\F(\e) \prec \F(\e'_0)$. Since $\gamma''_k \cap B = \emptyset$, none of $\F(\e'_k)$ belongs to $\acc$. By Lemma \ref{lem:goodleadlinessequence}, $\e'_k \to -\infty$ so $\F(\e)$ belongs to $(\F(\e'_{i+1}), \F(\e'_i))$ for some $i \ge 0$. Let us now prove that this implies that $B \subset U_i$. Indeed, take a lead line $\gamma$ that determines $\e$ such that $\gamma \cap B = \emptyset$, so $F(\gamma) \cap F(B) = \emptyset$. Notice that the arcs $\gamma''_{i+1}, \gamma''_i$ only meet $F(B)$ in a final subarc or in their landing points. Define $\gamma'$ as $F(\gamma)$ if it is disjoint to $\gamma''_{i+1} \cup \gamma''_i$ or as a concatenation of an initial subarc of $\gamma''_{i+1}$ or $\gamma''_i$ and a final subarc of $F(\gamma)$ disjoint to $\gamma''_{i+1} \cup \gamma''_i$ (the choice depends on which of $\gamma''_{i+1}, \gamma''_i$ the arc $F(\gamma)$ meets for the last time). Clearly, $\gamma'$ is a lead line that defines $\F(\e)$ thus it meets $B$. However $\gamma'(0) \in [\gamma''_{i+1}(0) ,\gamma''_i(0)]$ and $\gamma' \cap F(B) = \emptyset$. We then conclude that $\gamma' \subset \overline{U_i}$ and, in particular, $B \cap U_i \neq \emptyset$.
Finally, the fact that $B \cap U_i \neq \emptyset$ implies that $B \subset U_i$ because the boundary of the domain $U_i$ is disjoint to $B$.

Now, it is easy to check that $\F^{-1}(\e'_{i+1})$ is a lower bound of $(\L(U) \setminus \acc) \cap \F^{-1}(\acc)$. For if $\e \in \L(U) \setminus \acc$, $\F(\e) \in \acc$ and $\F(\e) \prec \e'_{i+1}$ we can a take a lead line $F(\gamma)$ that defines $\F(\e)$ and is disjoint to $F(B)$. After a suitable modification (as the ones explained above) we can suppose that the lead line is disjoint to $U_i$ (in fact, we can suppose it is contained in $\overline{U_j}$ for some $j > i$). However, this contradicts the fact that $\F(\e) \in \acc$ because $B$ is contained in $U_i$.
\medskip

\begin{figure}[htb]
\begin{center}
\includegraphics[scale = .7]{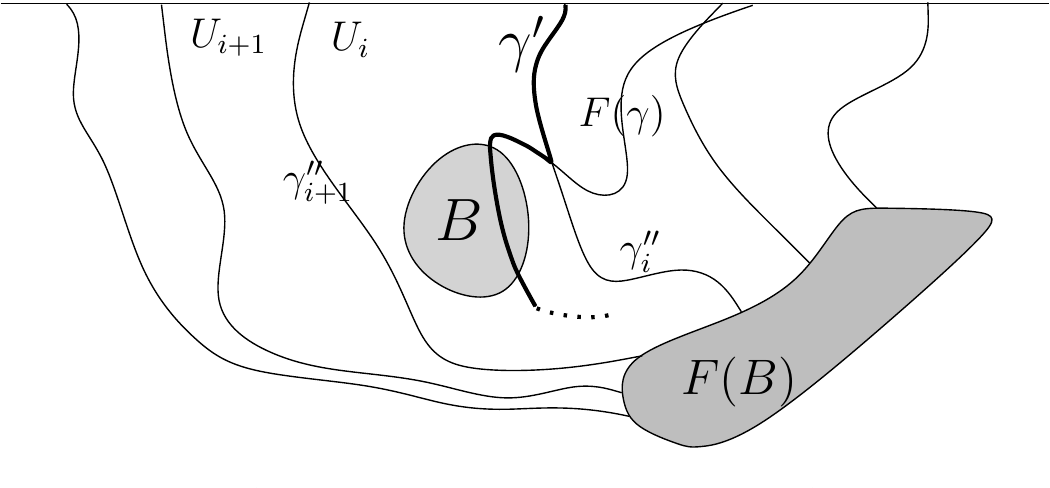}
\end{center}
\caption{Construction of $\gamma'$ in the claim in the proof of Proposition \ref{prop:notinfront}.}
\label{fig:rejilla}
\end{figure}

The conclusion now follows from an easy argument. Let $\gamma$ be a lead line landing at $q$ that determines the prime end $\q$.
Since $\gamma$ and $B$ are compact, we can choose an element from the increasing sequence $\{m_i\}$ supplied by (H2), say $m_j$, such that
$\T^{-m_j}(\gamma)$ and $B$ are disjoint and $\T^{-m_j + 1}(\q)$ is smaller than the lower bound of
$(\L(U) \setminus \acc) \cap \widehat{F}^{-1}(\acc)$. This implies that $\T^{-m_j}(\q) \in \L(U) \setminus \acc$
and, according to the claim, $\F(\T^{-m_j}(\q)) \in \L(U) \setminus \acc$.
Since $\rho \le 0$, $\F\T^{-m_j}(\q) \prec \T^{-m_j + 1}(\q)$ and we can apply the claim again to deduce that $\F^2\T^{-m_j}(\q)$ also belongs to $\L(U) \setminus \acc$. Following this procedure inductively we conclude that $\F^n \T^{-m_j}(\q)$ belongs to $\L(U) \setminus \acc$ for every $n \geq 0$.
In particular, if we apply this result to $n_j$, where $n_j$ comes from (H2), we obtain $\F^{n_j}\T^{-m_j}(\q) = \T^{-m_j}\F^{n_j}(\q) \in \L(U) \setminus \acc$, which is a contradiction
because  $T^{-m_j}F^{n_j}(q)$, the principal point of $\T^{-m_j}\F^{n_j}(\q)$, lies in $B$.
\end{proof}

The result obtained yields a precise statement that expresses the vague idea of $F(B)$ lying deeper than $B$.
The conclusion of Theorem \ref{thm:main} follows from the next proposition.

\begin{proposition}\label{prop:haciaatrasrotanigual}
For any accessible point $p \in \X$ with bounded backward semi--orbit there exists $C > 0$ such that
$$\big|(F^n(p))_1 -(p)_1 - n \rho \big| < C$$
for any $n \le 0$.
\end{proposition}

\begin{proof}
Proposition \ref{prop:notinfront} provides $\c \in \L(U)$ such that $\F(\acc_B \cap (-\infty, \c]) \subset \acc$.
In the following, denote $\acc \cap (-\infty, \c]$ by $\acc^{\c}$.
Our first step is to prove that the previous inclusion implies that $\acc^{\c}$ is positively invariant under $\F$, $\F(\acc^{\c}) \subset \acc^{\c}$.
Indeed, suppose that $\e$ belongs to $\acc^{\c}$ and let $\gamma$ be a lead line that determines $\e$.
Let $W$ be a neighborhood of $B$ such that $F(W) \cap B = \emptyset$.
Applying Lemma \ref{lem:goodleadlines} to $\e, \gamma$ and $W$ we obtain a prime end $\e' \preceq \e$ whose principal point lies in $B$,
so $\e' \in \acc_B^{\c}$, and a lead line $\gamma'$ that determines $\e'$ and is contained in $W \cup \gamma$.
Then, the arc $F(\gamma')$ determines $\F(\e') \in \acc^{\c}$ and is contained in $F(\gamma) \cap F(W)$.
Since $F(W)$ is disjoint to $B$ it follows that $F(\gamma)$ meets $B$. As the choice of $\gamma$ was arbitrary, $\F(\e) \in \acc^{\c}$.

Let $p$ be an arbitrary accessible point with bounded $\O^-(p)$ and $\p$ an accessible prime end such that $\Pi(\p) = \{p\}$.
Observe first that if $m$ is large enough, say $m \ge m_0 > 0$, $\p$ can be accessed with a lead line that does not meet $T^m(B)$ or,
equivalently, there is a lead line that determines $\T^{-m}(\p)$ and does not intersect $B$, i.e., $\T^{-m}(\p) \notin \acc$.

Recall that the sequence of prime ends $\{\q_i\}$ defined in the beginning of the proof of Theorem \ref{thm:main} satisfies $\q_i \to -\infty$. First, choose an element from the sequence which is smaller than $\min\{\c, \T^{-m_0}(\p)\}$ and an arbitrary lead line that determines it and apply Lemma \ref{lem:goodleadlines} to them (the prescribed neighborhood of $B$ is irrelevant). We obtain a lead line $\gamma_0$ that determines a prime end $\e_0 \prec \min\{\c, \T^{-m_0}(\p)\}$ whose principal point lies in $\partial B$. Then, take an element $\q^*$ of the sequence $\{\q_i\}$ that is smaller than $\T^{-1}(\e_0)$ and an associated lead line $\gamma^*$ that is disjoint to $\gamma_0$ except possibly for their landing point. Apply Lemma \ref{lem:goodleadlines} to $\q^*$ and $\gamma^*$ (the choice of $W$ is again irrelevant) to obtain $\e_1$ and $\gamma^*_1$. Since both $\gamma_0$ and $\gamma^*_1$ only intersect $B$ in a final subarc or in a singleton and $\e_1 \neq \e_0$, the lead lines $\gamma_0, \gamma^*_1$ do not share a common final subarc. Then, we can modify $\gamma^*_1$ by considering its last intersection point with $\gamma_0$ (in the case $\gamma_0, \gamma^*_1$ share a common landing point we exclude this point from our consideration), if any, and then replacing its initial subarc by an arc that goes parallel and disjoint to $\gamma_0$ (in a similar fashion as we did in Lemma \ref{lem:boundeddomain} to construct $\gamma_{\epsilon}$). We obtain a lead line $\gamma_1$ disjoint to $\gamma_0$ that determines $\e_1$.

The outcome of the procedure described in the previous paragraph is a pair of lead lines $\gamma_0, \gamma_1$ that define accessible prime ends $\e_0, \e_1$ whose principal points lie in $\partial B$ and satisfy:
$$
\e_1 \prec \T^{-1}(\e_0) \prec \e_0 \prec \min\{\c, \T^{-m_0}(\p)\}.
$$
\noindent
Furthermore, the intersection $\gamma_0 \cap \gamma_1$ is either empty or equal to a common landing point.
 As a consequence, in the first alternative there is an arc $\alpha$ in $\partial B$ so that we can apply Lemma \ref{lem:boundeddomain} to the triple $(\gamma_0, \gamma_1, \alpha)$ and in the second alternative we may employ Remark \ref{rmk:boundeddomain} to $(\gamma_0, \gamma_1)$.

For every $n \le 0$, there is a translation of $\F^{n}(\p)$, say $\T^{-a_n}\F^{n}(\p)$ that belongs to $[\e_1, \e_0]$.
Note that for every $n \le 0$,
$a_n \ge m_0$ because $\T^{-m_0}(\p) \preceq \T^{-1}(\p) \prec \F^{n}(\p)$ (recall $\rho \le 0$).
From Lemma \ref{lem:boundeddomain} (or Remark \ref{rmk:boundeddomain}) we obtain $C'> 0$ so that if $\e \in [\e_1, \e_0]$ is accessible and $(e)_1 \notin [-C', C']$,
where $e$ is the principal point of $\e$, then any lead line that determines $\e$ meets $\alpha$.
As a consequence, $|(e)_1| > C'$ automatically guarantees $\e \in \acc$.
This implies, in particular, that

\smallskip
\centerline{$|(T^{-a_0}(p))_1| = |(p)_1 - a_0| < C'$ because $\T^{-a_0}(\p) \notin \acc$.}
\smallskip

For any $n \le 0$,
\begin{align}
\label{eq:1} \big| (F^n(p))_1 - (p)_1 - n \rho \big| & = \big| (T^{-a_n}F^n(p))_1 + a_n - (p)_1 - n \rho \big| \le \\
& \le \big| (T^{-a_n}F^n(p))_1 \big| + \big|a_n - (p)_1 - n \rho \big| \le \nonumber \\
& \le \big| (T^{-a_n}F^n(p))_1 \big| + \big|(p)_1 - a_0 \big| + \big|a_n - a_0 - n \rho \big| \nonumber
\end{align}

Note that the rightmost term is bounded by $l+1$, where $l$ denotes the smallest integer such that $\T^{-l}\e_0 \prec \e_1$.
According to the arguments above, $a_n \ge m_0$ so $\T^{-a_n}(\p) \notin \acc$ and, in particular, $\T^{-a_n}(\p) \notin \acc^{\c}$.
From the positive invariance of $\acc^{\c}$ we deduce that if $n \le 0$ then
$\F^n\T^{-a_n}(\p) = \T^{-a_n}\F^n(\p) \notin \acc^{\c}$, so $|(T^{-a_n}F^n(p))_1| \le C'$.
Writing $C = 2C' + l + 1$, (\ref{eq:1}) yields \[|(F^n(p))_1 - (p)_1 - n \rho| < C,\] valid for all $n \le 0$.
\end{proof}

This ends the proof of Theorem \ref{thm:main}.
\end{proof}

As was briefly mentioned before stating (H1)--(H2), the case in which
$\{(F^n(q))_1 - n\rho\}_{n \ge 0}$ is not bounded from below is basically equal to the case
in which the sequence is unbounded from above that has been already addressed.
Indeed, the discrepancy is simply the result of the possible choices of the identification between $\AA$, the universal
cover of $\A$, and $\R \times (-\infty, 0]$.
As far as asymptotics is concerned, there are only two possibilities for
the covering projection of the angular coordinate, either $\widetilde{\theta} \mapsto e^{2\pi i\widetilde{\theta}}$
or $\widetilde{\theta} \mapsto e^{-2 \pi i\widetilde{\theta}}$.
Switching the choice of projection is equivalent to swapping ends of $\R$.

The analogues of the hypothesis (H1)--(H2) in this case are the following:

\smallskip
\noindent
There exist integers $a, b$ such that $b \ge 1$,
\begin{itemize}
\item[(H1')] $\rho \ge a/b =: \rho'$ and $\liminf_{n \to +\infty} (F^n(q))_1 - n \rho' = - \infty$ and
\item[(H2')] there are strictly increasing sequences of integers $\{m_i\}, \{n_i\}$ and an integer $k$, $0 \le k < b$, such that
$T^{m_i - an_i} F^{bn_i + k}(q)$ converges to a point $q_0 \in \X$ that is not fixed under $T^{-a} F^b$.
\end{itemize}

\medskip

It is easy to recast Lemmas \ref{lem:equalasymptotics}, \ref{lem:linealdriftH1H2} and \ref{lem:allperiodic}
 under these assumptions. We explicitly include the counterpart of Theorem \ref{thm:main}:

\begin{theorem}\label{thm:main2}
Suppose there exists an accessible point $q \in \X$ with bounded forward semi--orbit and such that
$$
\liminf_{n \to +\infty} \, (F^n(q))_1 - n \rho = - \infty,
$$
where $\rho = \widehat{\rho}(F,X)$, and (H1') and (H2') are satisfied.
Then, for any accessible point $p \in \X$ with $\O^-(p)$ bounded there exists $C > 0$ such that
$$\big|(F^n(p))_1 -(p)_1 - n \rho \big| < C$$
for any $n \le 0$.
\end{theorem}

Similar statements concerning drifting bounded backward semi--orbits instead of forward semi--orbits are easily obtained by applying these
results to $F^{-1}$. Let us now restate and prove Theorem \ref{thm:resumen} (\ref{item:2}).

\begin{theorem}\label{thm:resumen2}
Suppose that the sequences $\{(F^n(p))_1 - n \rho\}_{n \le 0}$
and $\{(F^n(q))_1 - n \rho\}_{n \ge 0}$ are unbounded for some accessible points $p, q \in \X$ such that $\O^-(p)$ and $\O^+(q)$ are bounded. Then:
\begin{enumerate}
\item \label{item:aa} Both sequences $\{(F^n(p))_1\}_{n \le 0}$ and $\{(F^n(q))_1\}_{n \ge 0}$ have the form $n\rho + o(|n|)$.
Equivalently, the backward rotation number of $p$ is equal to the forward rotation number of $q$ and they are both equal to $\rho$:
$$\lim_{n \to -\infty} (F^n(p))_1/n = \lim_{n \to +\infty} (F^n(q))_1/n = \rho.$$
\item \label{item:bb} If $\rho \in \Q$ then the closure of both the projection onto $\A$ of $\O^-(p)$ and $\O^+(q)$
contain a periodic point whose rotation number is $\rho$.
\item \label{item:cc} If $\rho \in \Q$ and $\{(F^n(q))_1 - n \rho\}_{n \ge 0}$ converges either to $-\infty$ or to $+\infty$
then any limit point of the projection onto $\A$ of the forward orbit of $q$ is periodic and has rotation number $\rho$.
An analogous statement holds for $p$ and the projection of $\O^-(p)$.
\end{enumerate}
\end{theorem}
\begin{proof}
Let us start with (\ref{item:aa}) and proceed by contradiction. Suppose that the forward rotation number of $q$ is not equal to $\rho$.
Argue with $F^{-1}$ instead of $F$ to conclude the other equality.
This means that

\smallskip
\centerline{either \enskip $\liminf_{n \to +\infty} (F^n(q))_1/n < \rho$ \enskip or \enskip $\rho < \limsup_{n \to +\infty} (F^n(q))_1/n$.}
\smallskip

\noindent
Notice that both the limit inferior and superior exist and are finite because they are bounded by the
minimum and maximum of the displacement function on $\mathcal{O}^+(q)$, respectively.
In any case, Lemma \ref{lem:linealdriftH1H2} or its analogue imply (H1)--(H2) or (H1')--(H2') hold.
Then, Theorems \ref{thm:main} and \ref{thm:main2} conclude that the drift of the backward orbit of $p$ is bounded, a contradiction.

Items (\ref{item:bb}) and (\ref{item:cc}) are addressed in a similar fashion.
The fact that $\{(F^n(p))_1 - n \rho\}_{n \le 0}$ is unbounded means the hypothesis of Theorems \ref{thm:main} and \ref{thm:main2}
are not fulfilled.
However, since $\{(F^n(q))_1 - n \rho\}_{n \ge 0}$ is unbounded we have that

\smallskip
\centerline{either \enskip $\limsup_{n \to +\infty} (F^n(q))_1-n\rho = +\infty$
\enskip or \enskip $\liminf_{n \to +\infty} (F^n(q))_1-n\rho = -\infty$.}
\smallskip

\noindent Moreover, $\rho \in \Q$ automatically guarantees that (H1) or (H1'), respectively, holds.
The only possibility left is that (H2) or (H2'), respectively, is not satisfied.
This is the case if we forbid the existence of a periodic point of rotation number $\rho$ in the closure of the projection
of $\mathcal{O}^+(q)$. The variation stated in (\ref{item:cc}) is similarly settled using Lemma \ref{lem:allperiodic}.

\end{proof}

Finally, let us recall and prove Corollary \ref{cor:rhorealizado}:

\begin{corollary}
Let $f \colon \overline{\A} \to \overline{\A}$ be a homeomorphism isotopic to the identity of the closed annulus. As usual, denote $X$ an invariant continuum such that $\overline{\A} \setminus X$ is homeomorphic to $\A$, $F$ and $\X$ lifts of $f$ and $X$ to the universal cover and $\rho$ the prime end rotation number of $F$ in $X$.

Then, there exists a point $q$ in the boundary of $\X$ whose forward rotation number is equal to $\rho$. Furthermore, there exist an ergodic $f$--invariant Borel probability measure supported in $\partial X$ whose rotation number is $\rho$.
\end{corollary}
\begin{proof}
Let us begin with the first statement and proceed by contradiction. As the dynamics is placed in the closed annulus every semi--orbit is bounded. Suppose that $q_0 \in \X$ is an accessible point for which the statement does not hold. Without lose of generality suppose that

\centerline{$\displaystyle \limsup_{n \to +\infty} \frac{(F^n(q_0))_1}{n} > \rho$.}

\noindent By Lemma \ref{lem:linealdriftH1H2}, we are in the conditions to apply Theorem \ref{thm:main} and conclude that the deviation of every backward semi--orbit of an accessible point $p \in \X$ is bounded, i.e., there exists $C > 0$ such that

\centerline{
$\big|(F^n(p))_1 -(p)_1 - n \rho \big| < C$, \enskip \enskip for every $n \le 0$.}

Suppose $p_0 \in \X$ is accessible and satisfies the previous inequalities for a constant $C_0$. Take any lift $q$ of a point that belongs to the $\alpha$-limit set of the projection of $p_0$ to the annulus. Then $q = \lim_{k \to +\infty} T^{-m_k}F^{n_k}(p_0)$, where $\{n_k\}_{k \ge 1}$ is a decreasing integer sequence.

\medskip

\noindent
\emph{Claim:} $q$ satisfies the conclusion in the statement.

Indeed, for a given $m \ge 0$ consider a sufficient large index $k$ so that $m + n_k < 0$ and
\[
\big| (T^{-m_k}F^{n_k}(p_0))_1 - (q)_1 \big| < 1 \enskip \enskip \text{and} \enskip \enskip \big| (T^{-m_k}F^{m + n_k}(p_0))_1 - (F^m(q))_1 \big| < 1
\]
In consequence,
\begin{align*}
\big| (F^m(q))_1 - & (q)_1 - m \rho \big| \le \big| (F^m(q))_1 - (T^{-m_k}F^{m + n_k}(p_0))_1 \big| + \\
&+ \big| (T^{-m_k}F^{m + n_k}(p_0))_1 - (T^{-m_k}F^{n_k}(p_0))_1 - m \rho \big| + \big| (T^{-m_k}F^{n_k}(p_0))_1 - (q)_1\big| < \\
&< 2 + \big| (F^{m + n_k}(p_0))_1 - F^{n_k}(p_0))_1 - m \rho \big| < 2 + 2C_0
\end{align*}
and it follows that the forward semi--orbit of $q$ has bounded deviation and, in particular, its forward rotation number is equal to $\rho$.

Let us move now the stronger statement on the realization of $\rho$ as the rotation number of an ergodic measure on $\partial X = \partial U$. For simplicity, thicken the annulus so that $\partial U$ does not meet its lower boundary. Denote $C^+, C^-$ to the connected components of $\overline{\A} \setminus \partial U$ that contain the boundary components of the annulus and let $K = \overline{\A} \setminus (C^+ \cup C^-)$. Then, $K$ is an $f$--invariant essential continuum composed of $\partial U$ and all except two of its complementary domains.

A result of Handel \cite{handel} improved by Koropecki \cite{koropecki} proves that there is an ergodic invariant Borel probability measure $\mu$ supported on $K$ such that $\rho(F, \mu) = \rho$. In order the fulfill the conditions in the statement we need to guarantee that the support of the measure is contained in $\partial U$.

Suppose $D$ is a connected component of the complement of $\partial U$ such that $\mu(D) > 0$. Then $D$ contains a point $x$ that is recurrent and $\rho(F, x) = \rho$ (because both conditions are generic with respect to $\mu$, i.e. satisfied for $\mu$--almost every point, by Poincaré recurrence theorem and Birkhoff ergodic theorem, respectively).
It follows that $D$ is periodic and $\rho = \rho(F, \mu) \in \Q$. Using this additional property, $\rho \in \Q$, we can apply the detailed description of \cite{koropasseggi}: either there is a periodic point on $\partial U$ with rotation number $\rho$ (which gives automatically an ergodic measure of rotation $\rho$) or $\partial U$ is contained in the basin of a finite family of at least two topological ``rotational'' attractors and repellors whose boundaries belongs to $\partial U$. In the latter alternative, any ergodic measure supported on the boundary of any of these attractors or repellors whose interior is contained in $\overline{\A}$ has rotation number $\rho$.
\end{proof}

The first statement of the previous corollary does not hold if we require $q$ to be accessible. Indeed, consider the continuum composed of the union of two closed disks and one curve spiraling onto them (similar to Figure 1 in \cite{BG}, see also Figure 2 in \cite{koropasseggi} and the description therein). Let the dynamics be a rotation in each disk and be wandering in the curve (carrying points from the second to the first disk). Embed the construction in the annulus by identifying the fixed point of the first disk to the lower end. Then, it is easy to check that the prime end rotation number is 0 but every accessible point has non--zero forward rotation number equal to the rotation number of the boundary of the first disk.

\section*{Acknowledgements}
The author expresses its gratitude to the referee, whose patient work and comments on the original manuscript led to a massive improvement of the presentation of the results.

\end{document}